\documentclass{article}
\usepackage[utf8]{inputenc}

\usepackage{geometry}
\usepackage{cmbright}
\usepackage[OT1]{fontenc}

\usepackage{subcaption}
\usepackage{amssymb, amsmath, amsthm}

\usepackage{color}

\newcommand{\RomanNumeralCaps}[1]

\theoremstyle{plain}
\newtheorem{thrm}{Theorem}[section]
\newtheorem{lmm}[thrm]{Lemma}
\newtheorem{crllr}[thrm]{Corollary}
\newtheorem{prpstn}[thrm]{Proposition}

\theoremstyle{definition}

\newtheorem{rmrk}[thrm]{Remark}

\def\xdif{\,{\rm d}}

\DeclareMathOperator{\Div}{div}

\newcommand{\R}{\mathbb R}
\newcommand{\U}{{\bf U}}
\newcommand{\e}[1]{{\bf e}_{#1}}
\newcommand{\normal}{{\bf n}}
\newcommand{\source}{u_b}
\newcommand{\Ge}{{\widetilde G}}
\newcommand{\HDiv}{{H(\Div, \Omega)}}
\newcommand{\V}{\widetilde {\bf U}}

\newcommand{\normTraceS}{\gamma_{1,s}}
\newcommand{\normTraceB}{\gamma_{1,b}}
\newcommand{\trace}{\gamma_0}
\newcommand{\traceS}{\gamma_{0,s}}
\newcommand{\traceB}{\gamma_{0,b}}

\newcommand{\mG}{{\mathcal G}}
\newcommand{\mH}{{\mathcal H}}
\newcommand{\mD}{{\mathcal D}}
\newcommand{\mV}{{\mathcal V}}

\newcommand{\intdom}{\int_{\Omega}}

\newcommand{\intb}{\int_{\Gamma_b}}
\newcommand{\ints}{\int_{\Gamma_s}}

\newcommand{\PhiAlpha}{{\bm\Phi}}
\newcommand{\mGAlpha}{{\bm G}}

\usepackage{amssymb}
\usepackage{graphicx}
\usepackage{bm}
\graphicspath{{pictures/}}

\numberwithin{equation}{section}

\begin{document}

\title{Well-posedness and potential-based formulation for the propagation of hydro-acoustic waves and tsunamis}

\footnotetext[1]{RWTH Aachen University, Math 111810, Templergraben 55, 52062 Aachen, Germany} 
\renewcommand{\thefootnote}{\fnsymbol{footnote}} 
\footnotetext[1]{Corresponding author: dubois@eddy.rwth-aachen.de} 
\renewcommand{\thefootnote}{\arabic{footnote}} 

\footnotetext[2]{Inria, M3DISIM Team, Inria Saclay, 1 rue Honoré d’Estienne d’Orves, 91120 Palaiseau, France} 

\footnotetext[3]{Seismology Group, Institut de Physique du Globe de Paris, Université Paris Diderot, Sorbonne Paris Cité, 1 rue Jussieu, 75005 Paris, France}

\footnotetext[4]{Inria, Team ANGE, Inria Paris, 2 rue Simone Iff, 75012 Paris and Laboratoire Jacques-Louis Lions, Sorbonne Université, 75005 Paris, France} 

\footnotetext[5]{Laboratoire Jacques-Louis Lions, Sorbonne Université, 75005 Paris, France}

\author{
\renewcommand{\thefootnote}{\arabic{footnote}} 
Juliette Dubois \footnotemark{} 
\renewcommand{\thefootnote}{\fnsymbol{footnote}} 
\footnotemark[1]{}
\renewcommand{\thefootnote}{\arabic{footnote}} ,
Sébastien Imperiale \footnotemark{} , 
Anne Mangeney \footnotemark{} $^{, 4}$, 
Jacques Sainte-Marie \footnotemark{} $^{,5}$
}

%
%
\maketitle


\begin{abstract} 
We study a linear model for the propagation of hydro-acoustic waves and tsunami in a stratified free-surface ocean.
A formulation was previously obtained by linearizing the compressible Euler equations. 
In this paper, we introduce a new formulation written with a generalized potential.
The new formulation is obtained by studying the functional spaces and operators associated to the model. 
The  mathematical study of this new formulation is easier and the discretization is also more efficient than for the previous formulation. 
We prove that both formulations are well posed and show that the solution to the first formulation can be obtained from the solution to the second.
Finally, the formulations are discretized using a spectral element method, and we simulate tsunamis generation from submarine earthquakes and landslides.
\end{abstract}

\section*{Introduction}
We present and analyze a model describing the propagation of hydro-acoustic waves and tsunami.
Hydro-acoustic waves are acoustic waves propagating in water. They are increasingly used to investigate movements of the seabed, as they can complete information brought by seismic waves \cite{cecioni_tsunami_2014, gomez_near_2021, caplan-auerbach_hydroacoustic_2001}. 
Hydro-acoustic waves could improve early-warning systems of tsunamis generated by submarine earthquakes or landslides. 
To this aim, it is relevant to study models coupling acoustic waves and tsunamis.

The traditional model for hydro-acoustic waves and tsunamis combines a linear acoustic equation in the domain and the linearized free-surface equation of an incompressible fluid \cite{ewing_proposed_1950}. 
Another linear model based on an irrotational flow assumption was obtain with an Eulerian-Lagrangian approach \cite{longuet-higgins_theory_1950}.
However, the mathematical analysis for those models is lacking.
In particular, it is not clear whether the systems preserve an energy. 
Energy preservation is a key element for ensuring stable numerical schemes. 
In a previous work \cite{dubois_acoustic_2023} we derived a new formulation from the compressible Euler equation in Lagrangian coordinates. 
The resulting model consists in a wave-like equation and the system preserves an energy under a realistic condition on the background stratification.

The system can be seen as a particular case of the Galbrun equation with a vanishing mean flow and a non-homogeneous boundary condition of Dirichlet type. 
Most studies on the Galbrun equation focus on the harmonic regime, and the functional framework for its well-posedness is still subject to studies (see the review paper \cite{maeder_90_2020}).
The literature on the analysis of the evolution problem is scarcer, but we can mention \cite{bonnet-ben_dhia_regularisation_2006, hagg_well-posedness_2021}. 
The Galbrun equation with a uniform mean flow and with homogeneous boundary conditions is studied in \cite{bonnet-ben_dhia_regularisation_2006} using a regularization method.
In \cite{hagg_well-posedness_2021}, it is shown that the solution to the Galbrun equation can be deduced from the solution to the linearized Euler equation. 

In this work, we show existence and uniqueness for the velocity field solution to the linear compressible Euler equations with gravity, non-homogeneous boundary conditions and no mean flow. 
Moreover we propose a set of PDE, involving two scalar potentials, that can be solved instead of finding a solution to the linearized compressible Euler equations. The mathematical study is easier  and the discretization more efficient that the velocity-field formulation.
Finally, we present several numerical results and simulation to assess the properties of the introduced models. 
The simulations reproduce hydro-acoustic waves and tsunamis generated by earthquakes and submarine landslides. We show that the interference pattern characteristics of landslide-generated hydro-acoustic waves is reproduced by the model, providing a unique tool to detect and characterize the landslide source. \\

In more details, the article is organized as follows. We first continue this introduction by recalling briefly {\bf the velocity-field formulation} obtained in \cite{dubois_acoustic_2023} and present the new formulation (denoted here  {\bf potential-based formulation}) just after. In Section \ref{sec:preliminary}, the functional spaces and the operators used for studying the two mentioned formulations are introduced. 
In Section \ref{sec:Potential}, we study the {\bf  the potential-based formulation}. Even though this formulation is new,  we start with this one because the analysis is much more direct than the analysis of {\bf the velocity-field formulation}. 
The latter is then studied in Section \ref{sec:Velocity}. The main difficulty of this section lies in finding an adapted lifting operator for the boundary term. 
The correspondence between the two formulations is discussed in Section \ref{sec:Equivalence}; we prove that the solution to the {\bf the velocity-field formulation} can be computed using the solution to {\bf  the potential-based formulation} (the reciprocal is formally true but of lesser interest and is not studied in detail). 
In Section \ref{sec:Approximation} we describe how the two formulations are discretized using spectral finite elements \cite{cohen_higher-order_2001, komatitsch_introduction_1999} and present several numerical experiments. 
A first set of simulations helps validate the model by comparison with the literature. 
It also illustrates the correspondence between the velocity-field formulation and the potential-based formulation.
A second set of simulations investigates the classical hypothesis of irrotational flow using the correspondence between the velocity-field and potential-based formulations.
The last set of simulations is a preliminary work towards the simulation of hydro-acoustic waves generated by landslides. We show that the proposed model is able to recover an interference pattern in the simulated hydro-acoustic waves, characteristic of the Lloyd-mirror effect that occurs during submarine landslides.

Readers more interested in the applications than in the mathematical details regarding the {\bf potential-based formulation} formulation could only read this introductory section and directly go to Section \ref{sec:Approximation}.

\subsection*{Velocity-field formulation}
 
We denote by $\Omega$ the domain representing an ocean at rest.
The coordinates of $\Omega$ are written $({\bf x}, z)$, with the horizontal coordinate ${\bf x}~\in~\R^{d-1}$, $d=2$ or $d=3$.
The domain is unbounded in the horizontal direction and bounded in the vertical direction, with a fixed surface at $z=H$ and a time-independent topography $z_b({\bf x})$ at the bottom, see Figure~\ref{fig:Domain}. 
The domain is written 
\[ \Omega = \{ ({\bf x}, z) \ | \; {\bf x} \in \R^{d-1}, \;  z_b({\bf x}) \leq z \leq H \},
\] 
and its boundary is denoted $\Gamma = \partial \Omega.$
The topography is assumed to satisfy the following conditions: ${  z_b \in W^{1,\infty}(\R^{d-1}) }$ and there exists positive scalars $(H_-, R_+)$ such that, 
\[
   0 \leq  z_b({\bf x}) \leq H_- < H  \quad   \mbox{ and }   \quad   \nabla z_b({\bf x}) = 0  \mbox{ for }  |{\bf x}| \geq R_+.  
\]
These properties ensure that the domain $\Omega$ does not degenerate and is Lipschitz. The surface and bottom boundaries are respectively
\[
{\Gamma_s = \{ ({\bf x},H ) \ | \ {\bf x} \in \R^{d-1} \} }
\quad  \text{and} \quad 
{\Gamma_b = \{ ({\bf x},z ) \ | \ {\bf x} \in \R^{d-1}, \; z=z_b({\bf x}) \} }.
\]
The considered set of partial differential equation has been introduced in \cite{dubois_acoustic_2023}; it consists in a linear system of equations for the fluid velocity $ \U({\bf x}, z, t)$ and reads
\begin{equation} \label{eq:Intro:V_Vol}
    \rho_0 \frac{\partial^2 \U}{\partial t^2} 
- \nabla \left(
	\rho_0 c_0^2\nabla \cdot \U
	- \rho_0 g \U \cdot  \e{z} \right)
- \nabla \cdot(\rho_0 g \U) \,  \e{z}
= 0, \quad \text{ in } \Omega \times [0,T].
 \end{equation}
Equation \eqref{eq:Intro:V_Vol} represents the propagation of the acoustic, internal gravity and surface gravity waves in a  stratified fluid. The vector $  \e{z} = (0\ 0\ 1)^t$  is the unit vector along the z-axis. The constant scalar $g>0$ is the acceleration of gravity.
The parameters $\rho_0$ and $c_0$ are respectively the fluid density at equilibrium and the sound speed. They depend on $z$ only, and we make the following assumptions, 
\begin{equation} 
\rho_0 \in C^1([0,H]) \quad  \mbox{ and } \quad  c_0  \in C^0([0,H]). 
\end{equation}
For the well-posedeness of the problem, the velocity and the density must satisfy some usual non-degeneracy properties. 
We assume that there exists positive scalars $(\rho_-, \rho_+)$ and $(c_-, c_+)$ such that
\[
  \rho_- \leq  \rho_0(z) \leq \rho_+ \quad \mbox{ and }\quad  \, c_- \leq  c_0(z) \leq c_+.
\]
Following the literature \cite{gill_atmosphere-ocean_1982}, we also defined the scalar field $N^2$ as
\begin{equation}\label{def_of_N}
 N^2(z) = - \frac{g}{\rho_0(z)}\frac{d \rho_0(z)}{dz} - \frac{g^2}{c_0^2(z)}.
\end{equation}
Such field is called the Brunt-Väisälä frequency, or buoyancy frequency. 
The case $N^2 < 0$ corresponds to a fluid that is denser above and lighter below, hence it is an unstable equilibrium. 
Since the equations were obtained by linearizing around a stable equilibrium, we assume in the following $N^2 \geq 0$.
The case $N^2=0$ corresponds to the case of a constant temperature, in which case the fluid is called barotropic.

Equation \eqref{eq:Intro:V_Vol} is completed with boundary conditions. 
On the bottom $\Gamma_b$ we consider a localized displacement of amplitude $ u_b $ of the seabed, caused for example by a submarine earthquake,  landslide or caldera collapse.
On the surface $\Gamma_s$ a stress-free boundary condition is applied. Here the stress is a pressure \cite{dubois_acoustic_2023}, and is proportional to the divergence of the velocity field; hence the boundary conditions read 
\begin{equation} \label{eq:Intro:V_BC}
\U \cdot \normal_b = \source
\quad \text{ on } \Gamma_b \times [0,T],
\qquad \nabla \cdot \U = 0  
\quad \text{ on } \Gamma_s \times [0,T],
\end{equation}
where $\normal_b$ is the outward unitary normal of the domain $\Omega$ on $\Gamma_b$. The regularity of the displacement $u_b$ will be stated later.
Finally, in our context it is relevant to choose vanishing initial conditions, which also simplifies the forthcoming analysis, 
\begin{equation}\label{eq:Intro:V_CI}
\U({\bf x}, z, 0) = 0,
\qquad 
\frac{\partial \U}{\partial t} ({\bf x}, z, 0)=0 \quad \text{ on } \Omega.
\end{equation}


\subsection*{Potential-based formulation} 
Problem \eqref{eq:Intro:V_Vol}-\eqref{eq:Intro:V_BC} can be written as an abstract wave equation using an unbounded linear operator $G$,
\begin{equation}
        \frac{d^2 \U}{d t^2} + G^* \Ge \U = 0,
\label{eq:Intro:Formal}
\end{equation}
where $G^*$ is the adjoint of $G$ and $\Ge$ is an extension of $G$. The introduction of this non-symmetric formulation using the extension $\Ge$ is motivated by the presence of a non-homogeneous essential condition in the boundary conditions (\eqref{eq:Intro:V_BC}). 
The expression \eqref{eq:Intro:Formal} will be obtained rigorously in Section \ref{sec:preliminary}. One originality of this work is to construct and analyze a ``dual'' or ``adjoint'' wave-like problem for a new unknown $ \Phi $, satisfying
\begin{equation}
    \frac{d^2 \Phi}{d t^2} +  \Ge G^* \Phi = 0.
\end{equation}
Such problem is shown to be equivalent -- in a sense given rigorously later -- to the problem \eqref{eq:Intro:Formal}. Moreover, it presents several advantages from mathematical and numerical perspectives. 
The new unknown $\Phi$ has three scalar components, $\Phi~=~(\varphi, \psi, \gamma)^t$,  where $ \varphi $ and $ \psi $ are scalar fields and $\gamma = \varphi_{|\Gamma_s}$. Later, we show that 
$\Phi$ is related to the velocity $\U$ by the formula
\begin{equation} \label{eq:Intro_UPhi}
    \U = G^\ast \Phi = - \nabla \varphi
+ N \left(\psi + \frac{N}{g} \varphi \right)  \e{z},
\end{equation}
hence it can be seen as a generalized potential.
We also show that $ \varphi $ and $ \psi $ satisfy a second set of coupled partial differential equations describing the same physical system as \eqref{eq:Intro:V_Vol}-\eqref{eq:Intro:V_BC}, more precisely
\begin{align}
\frac{\partial^2 \varphi}{\partial t^2} 
+ c_0^2 \nabla \cdot  \left(
 - \nabla \varphi
+ N \left(\psi + \frac{N}{g} \varphi \right)  \e{z} \right)
+ g \frac{\partial \varphi}{\partial z}
- g N \left(\psi + \frac{N}{g} \varphi \right)
= 0, 
\ \text{ in } \Omega \times [0,T],\label{eq:Intro:P_Vol}
\\
\frac{\partial^2 \psi}{\partial t^2} 
- N \frac{\partial \varphi}{\partial z}
+ N^2 \left(\psi + \frac{N}{g} \varphi \right) 
 = 0, \ \text{ in } \Omega \times [0,T]. \label{eq:Intro:P_Vol2}
\end{align}
The system is completed with boundary conditions. Using  \eqref{eq:Intro:V_BC} and \eqref{eq:Intro_UPhi}, we find on $\Gamma_b$  
\begin{equation}
 \U \cdot \normal_b
 =  - \nabla \varphi \cdot \normal_b
+ N \left(\psi + \frac{N}{g} \varphi \right) (  \e{z} \cdot \normal_b) = u_b, \ \text{ on } \Gamma_b \times [0,T].
\end{equation}
This condition is a Robin boundary conditions for the potentials $(\varphi, \psi)$. 
It is easier to consider than condition \eqref{eq:Intro:V_BC} for both the analysis and the discretization.
Finally, to obtain the boundary condition on $\Gamma_s$ we observe that  \eqref{eq:Intro:P_Vol} can be rewritten, using \eqref{eq:Intro_UPhi},
\[
\frac{\partial^2 \varphi}{\partial t^2} 
+ c_0^2 \nabla \cdot  \U
- g  \U \cdot  \e{z}
= 0, \quad \text{ in } \Omega \times [0,T].
\]
Formally, evaluating the equation above on the boundary $\Gamma_s$ and using   \eqref{eq:Intro:V_BC} we obtain a boundary condition involving a second-order time derivative that accounts for surface gravity waves, 
\begin{equation}
\frac{\partial^2 \varphi}{\partial t^2} 
- g  \U  \cdot \normal_s = 0, \ \text{ on } \Gamma_s  \times [0,T],
\end{equation} 
where $\normal_s$ is the outward unitary normal of the domain $\Omega$ on $\Gamma_s$.
The initial conditions are deduced from \eqref{eq:Intro:V_CI}, they read
\begin{equation}
\varphi({\bf x}, z, 0) 
= \frac{\partial \varphi}{\partial t} ({\bf x}, z, 0) 
= 0, 
\quad 
\psi({\bf x}, z, 0) 
=  \frac{\partial \psi}{\partial t} ({\bf x}, z, 0) = 0
\quad \text{ on } \Omega .
\label{eq:Intro:P_CI}
\end{equation}
When $d=3$, the system \eqref{eq:Intro:P_Vol} - \eqref{eq:Intro:P_CI} involves only two scalar fields ($\varphi$ and $\psi$), compared to three for the velocity based formulation \eqref{eq:Intro:V_Vol}. 
Moreover, this system is a generalization to the one commonly introduced in hydrodynamics \cite{gill_atmosphere-ocean_1982}.
Indeed, when considering the barotropic case, we have $N=0$ and the system of partial differential equations \eqref{eq:Intro:P_Vol}-\eqref{eq:Intro:P_Vol2} reduces to
\begin{align}
\frac{\partial^2 \varphi}{\partial t^2}
- c_0^2 \Delta  \varphi
+ g \frac{\partial \varphi}{\partial z}
= 0, \quad \text{ in } \Omega \times [0,T],
\label{eq:Intro:LH1}
\\
\frac{\partial^2 \psi}{\partial t^2}
 = 0, \quad \text{ in } \Omega \times [0,T].
\label{eq:Intro:LH2}
\end{align}
Because of the vanishing initial conditions, we have that $\psi=0$. The system is then described by the function $\varphi$ only. The boundary conditions are also simplified,
\begin{equation}
\frac{\partial^2 \varphi}{\partial t^2} 
+ g \nabla \varphi
\cdot \normal_s = 0, 
\quad \text{ on } \quad \Gamma_s  \times [0,T] \quad \mbox{ and }  \quad \nabla \varphi \cdot \normal_b
 = u_b, 
 \quad \text{ on } \quad \Gamma_b  \times [0,T].
\end{equation}
This simplified system corresponds to the system of equations that was already introduced in the literature \cite{dubois_acoustic_2023, longuet-higgins_theory_1950}.


%
\begin{figure}
\centering \includegraphics[height=0.3\hsize]{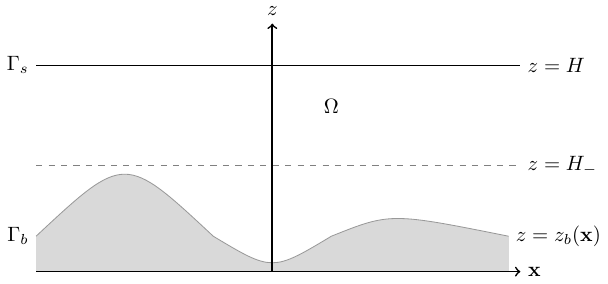}
\caption{The domain $\Omega$}
\label{fig:Domain}
\end{figure}
%
%



\section{Preliminary definitions}\label{sec:preliminary}
In this section, we define functional spaces and operators that will be used throughout the paper.


\subsection{Hilbert spaces and trace operators} 
We start by introducing the space of $H^1$-functions,
\[
 H^1(\Omega) = \{ \varphi \in L^2(\Omega) \ | \ \nabla \varphi \in L^2(\Omega)^d \},
\]
as well as the usual surjective trace operator
\[
\trace: H^1(\Omega) \to H^{1/2}(\Gamma),
\]
that is the extension to function in $H^1(\Omega) $ of the trace operator $\varphi \mapsto \varphi_{|\Gamma} $, defined for smooth functions.

The forthcoming analysis requires the use of the standard space of square integrable functions with  square integrable divergence,
\begin{equation*}
\HDiv = \{ \U \in L^2(\Omega)^d \ | \ 
\nabla \cdot \U \in L^2(\Omega)\}.
\end{equation*}
The space $\HDiv$ is equipped with the usual scalar product
\[
\forall \, (\U, \V) \in H(\Div, \Omega) \times H(\Div, \Omega), \quad (\U, \V)_\HDiv 
= (\U, \V)_{L^2(\Omega)^d}
+ (\nabla \cdot \U, \nabla \cdot \V)_{L^2(\Omega)}.
\]
In the following we use trace operators acting either on $\Gamma_s$ or on $\Gamma_b$ only, namely
\begin{align*}
 \traceS: H^1(\Omega) \to H^{1/2}(\Gamma_s), \quad  \traceS(\varphi) = \trace (\varphi)_{|\Gamma_s}, \\
 \traceB: H^1(\Omega) \to H^{1/2}(\Gamma_b), \quad  \traceB(\varphi) = \trace (\varphi)_{|\Gamma_b}.
\end{align*}
Since $\Gamma_s$ and $\Gamma_b$ are ``well-separated'' -- i.e., the distance between the two boundary is at least  $H-H_- $ that is strictly positive  --  one can also define surjective normal trace operators acting either on $\Gamma_s$ or on $\Gamma_b$ only. 
We introduce the normal trace operator
\[
\normTraceS :H(\Div, \Omega) \to H^{-1/2}(\Gamma_s).
\]
When applied to smooth functions, it corresponds to the operator
$ \normTraceS(\U)   = (\U \cdot \normal)_{|\Gamma_s}  $.
For functions in $\HDiv$ the operator $\normTraceS$ is defined as follows. 
We introduce a function $\chi$ depending on $z$ only, and satisfying 
\[
\chi \in C^1([0,H]), \quad
\chi(H)=1, 
\quad \text{and} \quad
 \chi(z) = 0 \text{ for } z < H_-.
\]
The normal trace operator on the surface is then defined by
\[
 \forall \, (\U, \varphi)  \in H(\Div, \Omega) \times H^1(\Omega), \quad
  \langle \normTraceS(\U) ,\traceS(\varphi)  \rangle_{\Gamma_s} 
  = ( \nabla \cdot (\chi \U) , \varphi)_{L^2(\Omega)} 
  +  ( \chi  \U , \nabla \varphi)_{L^2(\Omega)^{d}},
\]
where the duality product between  $H^{-1/2}(\Gamma_s)$ and $H^{1/2}(\Gamma_s)$ is denoted $\langle \cdot,\cdot \rangle_{\Gamma_s}$. 
In a similar way, we introduce $\normTraceB$, the normal trace operator acting on $\Gamma_b$ only. 
We denote by $\langle \cdot,\cdot \rangle_{\Gamma_b}$  the duality product between  $H^{-1/2}(\Gamma_b)$ and $H^{1/2}(\Gamma_b)$ and define the normal trace for functions in $H(\Div, \Omega)$ by 
\begin{multline*}
 \forall \, (\U, \varphi)  \in H(\Div, \Omega) \times H^1(\Omega), \quad 
 \langle \normTraceB(\U) ,\traceB(\varphi)  \rangle_{\Gamma_s} 
 = ( \nabla \cdot \big( (1-\chi) \U \big) , \varphi)_{L^2(\Omega)} 
 +  ( (1 - \chi)  \U , \nabla \varphi)_{L^2(\Omega)^{d}}.
\end{multline*}
Finally, for any function $\U \in H(\Div, \Omega) $, when stating that $ \normTraceS(\U)  \in L^2(\Gamma_s) $ we mean that $ \normTraceS(\U) $ is a function in $  H^{-1/2}(\Gamma_s) $ which can be identified with a function in $ L^2(\Gamma_s) $, and that the duality product reduces to the scalar product in  $L^2(\Gamma_s)$. 
This can be written as follows,
\[
\forall \, \U  \in H(\Div, \Omega), 
\quad  \normTraceS(\U)  \in L^2(\Gamma_s) \;
\Rightarrow  \;  \exists \, f  \in L^2(\Gamma_s) 
\; / \;  \langle \normTraceB(\U) ,\traceB(\varphi)  \rangle_{\Gamma_s} = \int_{\Gamma_s}  f \, \traceB(\varphi) \, \xdif s.
\]
When $\U$ is smooth, the function $f$ is given by $f  = (\U \cdot \normal)_{|\Gamma_s}. $


\subsection{The operator $G$ and its extension $\Ge$}
To introduce the operator associated to the evolution problem \eqref{eq:Intro:V_Vol} and the abstract wave equation \eqref{eq:Intro:Formal}, we introduce the Hilbert space
$
\mH = L^2(\Omega)^{d} 
$
equipped with the following weighted scalar product,
\begin{equation}
(\U, \V)_\mH = \intdom \rho_0 \U \cdot \V \xdif x.
\end{equation}
we also define the space $\mG$, 
\begin{equation}
\mG = L^2(\Omega) \times L^2(\Omega) \times L^2(\Gamma_s),
\end{equation}
equipped with the weighted scalar product
\begin{equation}
\forall \, \Phi   = \begin{pmatrix}
\varphi \\ \psi \\ \gamma
\end{pmatrix}  
\in \mG, \; \,  \forall \, \widetilde  \Phi =
  \begin{pmatrix}
\tilde \varphi \\ \tilde  \psi \\ \tilde \gamma
\end{pmatrix} \in \mG, \quad
(\Phi, \widetilde \Phi)_\mG
= \intdom \frac{\rho_0}{c_0^2} \varphi \tilde \varphi \xdif x 
+ \intdom \rho_0 \psi \tilde \psi \xdif x 
+ \ints \frac{\rho_0}{g} \gamma \tilde \gamma \xdif s.
\end{equation}
We first define the operator $\Ge$ used in Equation \eqref{eq:Intro:Formal}. The domain of $\Ge$ is denoted  $\mD(\Ge) \subset \mH$ and is defined by
\begin{equation}
\mD(\Ge) = \{ \U \in \HDiv 
\ | \  \normTraceS(\U)  \in L^2(\Gamma_s) \}.
\end{equation}
The operator $\Ge:\mD(\Ge) \subset \mH \to \mG$ is then defined by
\begin{equation}\label{def_Ge}
\forall \, U \in D(\Ge), \; \,
\Ge \, \U 
= \begin{pmatrix}
c_0^2
\left( \nabla \cdot \U 
	- \frac{g}{c_0^2} \U \cdot  \e{z} 
\right)
\\[4pt]
N \U \cdot  \e{z}
\\[4pt]
- g \normTraceS( \U)
\end{pmatrix}.
\end{equation}
It can be shown that the operator $\Ge$ is closed and densely defined.  
As already mentioned, it is useful to see the operator $\Ge$ as an extension of an operator $G$, defined on the domain $\mD(G) \subset \mH$, given by
\begin{equation}
\mD(G) =  \{ \U \in \HDiv 
\ | \  \normTraceS(\U)  \in L^2(\Gamma_s), \, \normTraceB(\U)  = 0 
\}.
\end{equation}
We have  $\mD(G) \subset \mD(\Ge)$ and by definition $ G $ satisfies, for all  
$ \U \in \mD(G), \; G \U = \Ge \U $, therefore
\begin{equation}\label{def_G}
\forall \, U \in D(G), \; \,
G\, \U 
= \begin{pmatrix}
c_0^2
\left( \nabla \cdot \U 
	- \frac{g}{c_0^2} \U \cdot  \e{z} 
\right)
\\[4pt]
N \U \cdot  \e{z}
\\[4pt]
- g \normTraceS( \U)
\end{pmatrix}.
\end{equation}
The operator  $ G $ is also closed and densely defined.


\subsection{The adjoint operators $G^*$ and $\Ge^{\, *}$ and a Green's formula}
Since the operators $ G $ and $\Ge $ are densely defined and closed, their adjoint -- denoted  $G^\ast$ and $\Ge^{\, *}$ respectively -- exist and are also densely defined and closed. 
We give their expression in this section. In the following, the space of smooth functions with compact support in $\Omega$ is denoted $\mD(\Omega)$ and $\mD(\overline{\Omega})$ corresponds to the space of smooth functions up to the boundary.

\begin{thrm} The operator $G^\ast:\mD(G^\ast) \subset \mG \to \mH$ is defined by 
\[
    \mD(G^\ast) 
= \{ \Phi=(\varphi, \psi, \gamma)^t \in \mG \ | \ \varphi \in H^1(\Omega), \gamma = \traceS( \varphi) \},
\]
and, for all $  \Phi=(\varphi, \psi, \gamma)^t  \in \mD(G^*) $, 
\begin{equation}\label{def_G_ast}
G^\ast 
\Phi
= - \nabla \varphi
+ N \left(\psi + \frac{N}{g} \varphi \right)  \e{z}.
\end{equation}
\end{thrm}

\begin{proof} 
Let $\U$ be a function in $\mD(\Omega)^d$, and let $\Phi~=~(\varphi \ \psi \ \gamma)^t$ belong to $\mD(G^\ast) $. It holds, by definition of the adjoint,
  \[
    (G \U, \Phi)_\mG = (\U, G^\ast \Phi)_\mH = (\U, \V)_\mH,
  \]
for some $ \V \in \mH.$ The equality above is developed using the definition of $G$, 
\begin{equation} \label{proof_adjoint}
    (G \U, \Phi)_\mG =   - \langle \nabla (\rho_0 \varphi),  \U  \rangle_{\Omega}  + 
     \intdom \rho_0   \Big(  N \psi - 
  \frac{g}{c_0^2}   
  \varphi \Big)  \U \cdot  \e{z}  \xdif x 
  = (\U, \V)_\mH,
\end{equation}
where  $\langle \cdot,\cdot \rangle_{\Omega} $ correspond to the duality product in  $\mD(\Omega)^d$. The equality \eqref{proof_adjoint} shows that $  \nabla (\rho_0 \varphi) $ belongs to $ L^2(\Omega)^3 $, hence, since $\rho_0$ is smooth, $\varphi \in H^1(\Omega)$. Equation \eqref{proof_adjoint} also shows that
  \[
    G^\ast \Phi = \V = -   \rho_0^{-1} \nabla (\rho_0 \varphi) +   \Big(  N \psi - 
  \frac{g}{c_0^2}    
  \varphi \Big)    \e{z}.    
  \]
  The simpler expression \eqref{def_G_ast} is obtained by, first, distributing the gradient, then, using the definition of the scalar field $N$. 
\end{proof}
 
The following Green formula holds
\begin{lmm}\label{GreenFormula}
For all $\U \in D(\Ge)$ and $ \Phi = (\varphi,\psi,\gamma)^t \in  D(G^*) $, it holds
\begin{equation}\label{eq:GF}
    (\Ge \U, \Phi)_{\mG} 
    =  (\U, G^* \Phi)_{\mH}  + \langle \normTraceB(\U), \, \traceB(\varphi)  
    \rangle_{\Gamma_b} .
\end{equation}
\end{lmm}
\begin{proof}
For $\U \in \mD(\overline \Omega)^d$ and $\Phi =  (\varphi,\psi,\varphi|_{\Gamma_s})^t $ with $\varphi $ and $\psi$  in $  \mD(\overline \Omega)$, we have  
\[
(\Ge \U, \Phi)_{\mG} 
=(\U, G^* \Phi)_{\mH}  +  \int_{\Gamma_b} \varphi_{|\Gamma_b}  \, (\U \cdot \normal_b)_{|\Gamma_b} \xdif s = \langle \normTraceB(\U), \, \traceB(\varphi)  
     \rangle_{\Gamma_b}.
\]
We conclude by using the density of $  \mD(\overline \Omega)^d $ in  $\HDiv$ and the density of $  \mD(\overline \Omega) $ in $ H^1(\Omega)$ and in $L^2(\Omega)$ (see \cite{girault_finite_1986} for details on these standard density results). 
\end{proof}

Thanks to the Lemma \ref{GreenFormula} we deduce the expression of $\Ge^{\, *}$.
\begin{crllr}
  The operator  $\Ge^{\, *}:\mD(\Ge^{\, *}) \subset \mG \to \mH$ is defined by
\begin{equation}
\mD(\Ge^{\, *}) 
= \{ \Phi=(\varphi, \psi, \gamma)^t \in \mG \ | \ \varphi \in H^1(\Omega), \  \traceS( \varphi) = \gamma, \  \traceB(\varphi) = 0\}  \subset \mD(G^\ast) ,
\end{equation}
and  for all $ \Phi \in \mD(\Ge^{\, *}) $, $ \Ge^{\, *} \Phi  =G^\ast \Phi $.
\end{crllr}

\begin{proof} 
Since $ \mD(G) \subset \mD(\Ge) $, we have the inclusion $ \mD(\Ge^{\, *}) \subset \mD(G^*) $ and  $ G^* $ is an extension of $ \Ge^{\, *}$. Therefore Lemma \ref{GreenFormula} can be used as follows: for all $\U \in D(\Ge)$ and $ \Phi = (\varphi,\psi,\gamma)^t \in  D(\Ge^{\, *}) $, it holds
\begin{equation}
    (\Ge \U, \Phi)_{\mG} 
     - (\U, \Ge^{\, *} \Phi)_{\mH} 
     = \langle \normTraceB(\U), \, \traceB(\varphi)  
     \rangle_{\Gamma_b},
\end{equation}
which shows that $ \langle \normTraceB(\U), \, \traceB(\varphi)  
     \rangle_{\Gamma_b} = 0. $ Using the surjectivity of the normal trace operator $ \normTraceB $, we deduce that $  \traceB(\varphi)   = 0.$

\end{proof}

The space $\mD(G^*)$ is equipped with the graph norm, 
\[
\| \Phi \|_{\mD(G^*)}^2
= \| \Phi \|_\mG^2 + \| G^*\Phi \|_\mH^2.
\]
And we have the following result,

\begin{prpstn} \label{prop:GraphNorm}
There exists a constant $C_c>0$ such that 
\[
\forall \Phi = 
\begin{pmatrix}
\varphi \\ \psi \\ \gamma
\end{pmatrix}
\in \mD(G^*),  \quad 
\| \Phi \|_{\mD(G^*)}
\geq C_c \| \varphi \|_{H^1(\Omega)}.
\]
\end{prpstn}

\begin{proof}
In the proof, we use the symbol $\lesssim$ for inequalities that hold up to a constant independent of $\Phi$. For $\Phi \in \mD(G^*)$, we have
  \[
    \| \varphi \|^2_{L^2(\Omega)^3}
    \lesssim
    \intdom \frac{\rho_0}{c_0^2} \varphi \xdif x 
    \leq \| \Phi \|^2_\mG
    \quad  \text{ and } \quad 
    \| \nabla \varphi \|^2_{L^2(\Omega)^d}
     \lesssim
    \| \nabla \varphi \|^2_\mH.
  \]
  It holds, by the triangular inequality,
\begin{equation*}
  \| \nabla \varphi \|^2_\mH
    \lesssim
\left \| - \nabla \varphi 
    + N (\psi + \frac{N}{g} \varphi) \e{z}
\right \|_\mH^2 
+ \left \| N (\psi + \frac{N}{g} \varphi) \e{z}
\right \|^2_\mH  ,
\end{equation*}
hence the norm of the gradient is bounded by $  \| G^* \Phi \|_\mH^2
+  \| \Phi \|_\mG^2 $, which concludes the proof.
\end{proof}

In this section, we have introduced all the necessary operators for the study of the potential-based problem \eqref{eq:Intro:P_Vol}-\eqref{eq:Intro:P_CI} and the velocity-field problem \eqref{eq:Intro:V_Vol}-\eqref{eq:Intro:V_CI}.




\section{Analysis of the potential-based formulation} \label{sec:Potential}

 This section is dedicated to the study of the problem \eqref{eq:Intro:P_Vol}-\eqref{eq:Intro:P_CI}. We show that the problem is well posed and that its solution satisfies an energy equality. We recall the system of PDE satisfied by the potentials $\varphi$ and $\psi$, 
\begin{align}
\frac{\partial^2 \varphi}{\partial t^2} 
+ c_0^2 \nabla \cdot  \left(
 - \nabla \varphi
+ N \left(\psi + \frac{N}{g} \varphi \right)  \e{z} \right)
+ g \frac{\partial \varphi}{\partial z}
- g N \left(\psi + \frac{N}{g} \varphi \right)
= 0, 
\ \text{ in } \Omega \times [0,T],\label{eq:P:PDE_Vol}
\\
\frac{\partial^2 \psi}{\partial t^2} 
- N \frac{\partial \varphi}{\partial z}
+ N^2 \left(\psi + \frac{N}{g} \varphi \right) 
 = 0,  \ \text{ in } \Omega \times [0,T],\label{eq:P:PDE_Vol2}
\end{align}
with boundary conditions
\begin{equation} \label{eq:P:PDE_Bot}
- \nabla \varphi \cdot \normal_b
+ N \left(\psi + \frac{N}{g} \varphi \right) (  \e{z} \cdot \normal_b) = u_b, \ \text{ on } \Gamma_b \times [0,T],
\end{equation}
and  
\begin{equation} \label{eq:P:PDE_Surf}
\frac{\partial^2 \varphi}{\partial t^2} 
- g  \left(
- \nabla \varphi 
+ N \left(\psi + \frac{N}{g} \varphi \right) \e{z}
\right) \cdot \normal_s = 0,
\ \text{ on } \Gamma_s  \times [0,T].
\end{equation} 
The system \eqref{eq:P:PDE_Vol}-\eqref{eq:P:PDE_Surf} is completed with vanishing initial conditions.


\subsection{Variational formulation}
The natural idea for writing the variational formulation associated to 
\eqref{eq:P:PDE_Vol}-\eqref{eq:P:PDE_Vol2} 
consists in testing  
\eqref{eq:P:PDE_Vol}-\eqref{eq:P:PDE_Vol2} against a function $(\tilde \varphi, \tilde \psi) \in H^1(\Omega) \times L^2(\Omega)$.
After integrating by parts, using the boundary conditions \eqref{eq:P:PDE_Bot}-\eqref{eq:P:PDE_Surf} and the definition of $N$ given by \eqref{def_of_N}, we obtain the problem:
given $u_b$ regular enough, 
find 
\begin{equation}
\begin{pmatrix}
\varphi \\ \psi
\end{pmatrix} \in L^2(0,T ; H^1(\Omega) \times  L^2(\Omega)) , 
\quad 
\frac{d}{dt}\begin{pmatrix}
\varphi \\ \psi
\end{pmatrix} \in L^2(0,T; L^2(\Omega)^2)
\label{eq:P:Non_standard_regularity},
\end{equation}
solution to
\begin{multline}
\frac{d^2}{dt^2} 
\intdom \frac{\rho_0}{c_0^2} 
\varphi \, \tilde \varphi \xdif x
+ \frac{d^2}{dt^2} 
\intdom \rho_0 \psi \, \tilde \psi \xdif x
\\
+ \intdom \rho_0
\left(
- \nabla \varphi + N \left( \psi + \frac{N}{g} \varphi \right)  \e{z}
\right)
\left(
- \nabla \tilde \varphi + N \left( \tilde \psi + \frac{N}{g} \tilde \varphi \right)  \e{z}
\right)
\xdif x
\\
+ \frac{d^2}{dt^2} \ints \frac{\rho_0}{g} 
\varphi \, \tilde \varphi \xdif s
+ \intb \rho_0 u_b \traceB(\tilde \varphi) \xdif s = 0,
\quad \forall
\begin{pmatrix}
\tilde \varphi \\ \tilde \psi
\end{pmatrix}
\in H^1(\Omega) \times L^2(\Omega).
\label{eq:P:Non_standard_weak}
\end{multline}
The formulation \eqref{eq:P:Non_standard_weak} will be useful for the numerical approximation. 
Indeed, the natural spaces for the  discretization are classical: $H^1(\Omega)$ for $\varphi$ and $L^2(\Omega)$ for $\psi$. Moreover, for $d=3$, the velocity formulation has three scalar unknowns whereas the potential formulation requires only two scalars unknowns. Finally, the source $u_b$ appears naturally as a Neumann condition in this formulation.


\subsection{Existence and uniqueness results}\label{sec:existence_anal}

The existence of a solution to \eqref{eq:P:PDE_Vol}-\eqref{eq:P:PDE_Vol2} cannot be directly proved by standard methods (such as \cite{lions_non-homogeneous_1972}), because of the surface condition \eqref{eq:P:PDE_Surf} involving the second-order time derivative of $ \varphi$. 
Therefore, we need to introduce a new unknown to the problem, denoted $\gamma \in L^2(\Gamma_s)$, and we define the vector of unknowns
\begin{equation}
\Phi(t) = 
\begin{pmatrix}
\varphi(t) \\ \psi(t) \\ \gamma(t)
\end{pmatrix}
\in \mG.
\end{equation} 
From the variational formulation \eqref{eq:P:Non_standard_weak} we deduce that the problem  \eqref{eq:P:PDE_Vol}
-  \eqref{eq:P:PDE_Surf}
reduces to the following abstract formulation: 
assume $u_b \in H^1(0,T;H^{-1/2}(\Gamma_b))$ given, find 
\[
\Phi \in L^2(0,T;\mD(G^*)), \quad 
\frac{d }{dt} \Phi \in L^2(0,T;\mG),
\]
solution to
\begin{align}
\frac{d^2}{dt^2}  ( \Phi(t), \tilde \Phi)_\mG
+ (G^* \Phi(t), G^* \tilde \Phi)_\mH
= \ell_b(t,\tilde \Phi),
& \quad \forall \, \tilde \Phi \in \mD(G^*),
\quad \text{ in }  \mD'(0,T),
\label{eq:P:Standard_weak}
\\
\Phi(0) = \frac{d}{dt}\Phi(0) = 0,
\label{eq:P:Standard_CI}
\end{align}
where $\ell_b:(0,T) \times \mD(G^*) \to \R$ is the linear form
\begin{equation}
\ell_b(t,\Phi) = \langle u_b(t), \traceB(\varphi) \rangle_{\Gamma_b}.
\end{equation}
Note that if $\Phi \in \mD(G^*)$, then $\gamma$ is the surface trace of $\varphi$, and the equation  \eqref{eq:P:Standard_weak} is exactly the equation \eqref{eq:P:Non_standard_weak}. 
For the formulation \eqref{eq:P:Standard_weak}-\eqref{eq:P:Standard_CI} we have the following result, 
\begin{prpstn} \label{prop:P:Existence}
Assume that $u_b \in H^{1}(0,T; H^{-1/2}(\Gamma_b))$.
Then the problem \eqref{eq:P:Standard_weak}-\eqref{eq:P:Standard_CI} has a unique solution and, up to a modification on zero measure sets, 
\[
 \Phi \in C^0 \big([0,T]; D(G^*) \big) \cap C^1\big([0,T];\mG \big). 
\]
\end{prpstn}

\begin{proof}
First we show that, for almost all $t \in (0,T)$, the form
$\ell_b(t)$ is a bounded linear functional on $\mD(G^*)$. 
Let $\Phi = (\varphi, \psi, \traceS(\varphi))^t \in \mD(G^*)$.
From the continuity of the trace operator, there exists a scalar $C_H > 0$ depending only on $H$, such that 
\[
|\langle \ell_b(t), \Phi \rangle| 
= |\langle  u_b(t), \traceB(\varphi)
\rangle_{ \Gamma_b) }| \leq C_H \|u_b(t) \|_{H^{-1/2}(\Gamma_b)}
\| \varphi \|_{H^1(\Omega)}, 
\]
and from Proposition \ref{prop:GraphNorm} we obtain
\[
|\langle \ell_b(t), \Phi \rangle|  \leq C_H C_c^{-1} \|u_b(t) \|_{H^{-1/2}(\Gamma_b)}   \| \Phi \|_{\mD(G^*)},
\]
hence $\ell_b(t)$ is bounded. From the continuity in time of $u_b$, we have
$\ell_b \in H^1(0,T; \mD(G^*)')$. 
From this property and \eqref{eq:P:Standard_weak}, we deduce that 
\[
    \frac{d^2 }{dt^2} \Phi \in L^2(0,T; \mD(G^*)'),
\]
hence the initial conditions \eqref{eq:P:Standard_CI} make sense. The existence and uniqueness of a solution  to the problem \eqref{eq:P:Standard_weak}-\eqref{eq:P:Standard_CI} follows then directly from standard results, see e.g.  \cite{lions_non-homogeneous_1972} and \cite{ joly_effective_2008}. 
\end{proof}


\subsection{Energy identity}
An energy identity for the system \eqref{eq:P:Standard_weak}-\eqref{eq:P:Standard_CI} is obtained following the usual approach \cite{lions_non-homogeneous_1972}.
We define the energy
\begin{equation}\label{def_energy}
\mathcal E(t) = \frac{1}{2} \big( \| \partial_t \Phi \|^2_\mG + \| G^* \Phi \|_\mH^2 \big). 
\end{equation}
Taking formally $\tilde \Phi = \partial_t \Phi$ in the weak formulation \eqref{eq:P:Standard_weak} yields
\[
  (\partial^2_{tt} \Phi, \partial_t \Phi)_\mG
+ (G^* \Phi, G^* \partial_t \Phi)_\mH
= \langle u_b , \traceB (\partial_t \varphi ) \rangle_{\Gamma_b},
\]
which is equivalent to 
\begin{equation}
\frac{d \mathcal E}{dt} 
= 
\frac{d}{dt} \langle u_b , \traceB (\varphi) \rangle_{\Gamma_b}
-  \langle \partial_t u_b, \traceB (\varphi) \rangle_{\Gamma_b}.
\end{equation}
Integrating the above equation from 0 to $t$ and using the vanishing initial conditions yields
\begin{equation} \label{eq:P:Energy}
  \mathcal E(t) 
= 
\langle u_b(t) , \traceB (\varphi(t) ) \rangle_{\Gamma_b}
-  \int_0^t \langle \partial_t u_b(s), \traceB (\varphi(s)) \rangle_{\Gamma_b} \xdif s.
\end{equation}
One can show that the identity  \eqref{eq:P:Energy} holds for the solutions given by Proposition \ref{prop:P:Existence} (see \cite{joly_effective_2008}). 
The inequality \eqref{eq:P:Energy} is the starting point to derive an estimate of the solution. 
We give below such estimate. 

\begin{prpstn} \label{prop:Energy} There exists $C>0$ such that, for any $u_b \in H^{1}(0,T; H^{-1/2}(\Gamma_b))$, the solution $\Phi$ to \eqref{eq:P:Standard_weak}-\eqref{eq:P:Standard_CI} satisfies,
\begin{equation}\label{eq:th:estimate}
    \sup_{s \in [0,t]} \mathcal E(s) 
    \leq  C \, (t^2 + 1) \, B^2(t), 
\end{equation}
Where $B$ is given by
\begin{equation}\label{eq:def_B}
B(t) = 
\sup_{s \in [0,t]}  \| u_b (s) \|_{H^{-1/2}(\Gamma_b)}
+  \int_0^t \| \partial_t u_b(s) \|_{H^{-1/2}(\Gamma_b)} \xdif s.
\end{equation}
\end{prpstn}

\begin{proof}
Starting with the equation \eqref{eq:P:Energy}, we have for the right-hand side
\begin{equation}\label{eq:energy_bound_proof}
\mathcal E(t)  = \langle u_b(t) , \traceB (\varphi(t) ) \rangle_{\Gamma_b}
-  \int_0^t \langle \partial_t u_b(s), \traceB (\varphi(s)) \rangle_{\Gamma_b} \xdif s
\leq 
 B(t) \, \sup_{s\in [0,t]} \| \traceB(\varphi)(s) \|_{H^{1/2}(\Gamma_b)}.
\end{equation}
To estimate the norm on $H^{1/2}(\Gamma_b)$, we use the continuity of the trace and Proposition \ref{prop:GraphNorm}, 
\begin{equation} \label{eq:proof:estimate_Trace}
\sup_{s \in [0,t]}\| \traceB(\varphi)(s) \|_{H^{1/2}(\Gamma_b)}
\lesssim
\sup_{s \in [0,t]}  \| \Phi(s) \|_{\mD(G^*)},
\end{equation}
where we use the symbol $\lesssim$ for inequalities that hold up to a constant independent of $u_b$, $\Phi$ and $t$.  
We show now that the graph norm in the right-hand side of \eqref{eq:proof:estimate_Trace} can be bounded by the energy.
From the definition of the scalar product in $\mD(G^*)$ and the energy \eqref{def_energy}, we have
\begin{equation}\label{eq:proof:estimate_PHi}
\| \Phi(t) \|_{\mD(G^*)}^2
=  \| \Phi(t)\|_\mG^2 
+ \| G^* \Phi(t) \|_\mH^2 
\leq \| \Phi(t) \|_\mG^2 
+ 2 \mathcal E(t) .
\end{equation}
Since the initial conditions vanish, it holds 
\begin{equation*}
\Phi(t) = \int_0^t \partial_t \Phi(r) \xdif r
\quad \Rightarrow \quad
\| \Phi(t) \|_\mG \leq \int_0^t \| \partial_t \Phi(r) \|_\mG \xdif r
\leq \int_0^t \sqrt{2 \mathcal E(r)} \xdif r.
\end{equation*}
Using this inequality to simplify \eqref{eq:proof:estimate_PHi}  yields
\[
\| \Phi(t) \|_{\mD(G^*)}^2
\leq 
\left( \int_0^t \sqrt{2 \mathcal E(r)} \, \text{d}r \right)^2
+ 2 \mathcal E(t)
\leq \left( t \sup_{r \in [0,t]} \sqrt{2 \mathcal E(r)} \right)^2
+ 2 \mathcal E(t).
\]
We finally obtain the bound
\[
\sup_{r \in [0,t]} \| \Phi(r) \|_{\mD(G^*)}^2
\leq 2(t^2 + 1) \sup_{r \in [0,t]}  \mathcal E(r), 
\]
hence \eqref{eq:proof:estimate_Trace}, \eqref{eq:energy_bound_proof} and the inequality just above give,
\begin{equation}
\mathcal E(t)
\lesssim
\sqrt{(t^2 + 1)} B
\sup_{r \in [0,t]} \sqrt{ \mathcal E(r) },
\end{equation}
which allows to deduce \eqref{eq:th:estimate} using the Young inequality.
\end{proof}

The well-posedness of the potential-based formulation \eqref{eq:P:Standard_weak} was obtained by standard tools \cite{lions_non-homogeneous_1972}. The study of the velocity-field formulation is more involved, and is the subject of the next section.




\section{Analysis of the velocity-field formulation} \label{sec:Velocity}

This section is devoted to prove an existence and uniqueness result for 
the system of partial differential equations for the velocity-field formulation. We recall the problem at hand: for $u_b$ given, find $\U$   solution to 
\begin{equation} \label{eq:V:Vol}
    \rho_0 \frac{\partial^2 \U}{\partial t^2} 
- \nabla \left(
	\rho_0 c_0^2\nabla \cdot \U
	- \rho_0 g \U \cdot  \e{z} \right)
- \nabla \cdot(\rho_0 g \U) \,  \e{z}
= 0, \quad \text{ in } \Omega \times [0,T],
 \end{equation}
with the boundary conditions
\begin{equation} \label{eq:V:BC}
\U \cdot \normal_b = \source
\quad \text{ on } \Gamma_b \times [0,T],
\qquad \nabla \cdot \U = 0  
\quad \text{ on } \Gamma_s \times [0,T],
\end{equation}
and with vanishing initial conditions.


\subsection{Variational formulation and uniqueness result} \label{sec:V:Uniqueness}
The variational formulation associated to the evolution problem \eqref{eq:V:Vol}-\eqref{eq:V:BC} is obtained by  testing the system \eqref{eq:V:Vol} against a function $\V$ and integrating over $\Omega$. The test function $\V$ is chosen such that its normal trace on $\Gamma_b$ vanishes. Using the boundary conditions \eqref{eq:V:BC} we obtain: 
\begin{multline}
\frac{d^2}{dt^2} \intdom \rho_0 \U(t) \V \xdif x + \intdom \rho_0 c_0^2 \left( \nabla \cdot \U(t) 
	- \frac{g}{c_0^2} \U(t) \cdot  \e{z} 
\right)   \left( \nabla \cdot \V 
	- \frac{g}{c_0^2} \V \cdot  \e{z} 
\right)  \xdif x  \\[8pt]
+  \intdom \rho_0 N^2 \U(t) \cdot  \e{z}  \V \cdot  \e{z}  \xdif x  +  \ints \rho_0 g  \U(t) \cdot \normal_s \,  \V \cdot \normal_s  \xdif s
= 0.
\end{multline}
The formulation above is completed with the non-homogeneous boundary condition  $\U \cdot \normal _b= \source $ on $\Gamma_b$. 
These formal computations show that the adequate variational formulation to study is the following: 
assume ${u_b \in H^2(0,T;H^{-1/2}(\Gamma_b))}$ given, and 
find 
\begin{equation}\label{eq:reg_U}
\U \in L^2(0,T;\mD(\Ge)), \quad 
\frac{d}{dt} \U \in L^2(0,T;\mH),
\end{equation}
solution to 
\begin{align}
\frac{d^2}{dt^2} (\U , \V )_\mH
+ (\Ge \U, G \V)_\mG
= 0, 
& \quad \forall \V \in \mD(G),
\quad \text{ in }  \mD'(0,T),
\label{eq:V:Weak_1}
\\
\normTraceB(\U) = u_b, 
& \quad \text{ in }   (0,T),
\label{eq:V:Weak_BC_1}
\\[4pt]
\U(0) =  \frac{d}{dt} \U(0) = 0.
\label{eq:V:Weak_CI_1}
\end{align}
Note that compared to Section \ref{sec:existence_anal} we have assumed slightly more regularity in time for the source term $u_b$. 
This is another drawback of the velocity-field formulation and is due to the nature of the condition: here the inhomogeneous boundary condition in  \eqref{eq:V:BC} is of essential type (similar to an inhomogeneous Dirichlet boundary condition). Of course one could weaken this regularity assumption, however, from our current analysis this would also weaken the regularity of the solution.  
It is rather direct -- using Lions-Magenes theory \cite{lions_non-homogeneous_1972} -- to prove the following result.

\begin{thrm} The solution to \eqref{eq:V:Weak_1}-\eqref{eq:V:Weak_CI_1} is unique. 
\end{thrm}
\begin{proof}
The problem  \eqref{eq:V:Weak_1}-\eqref{eq:V:Weak_CI_1} with $u_b \equiv 0$ amounts to 
find $\U $  with the regularity \eqref{eq:reg_U}
solution to 
\begin{equation}
\frac{d^2}{dt^2} (\U , \V )_\mH
+ (G \U, G \V)_\mG
= 0, 
\quad \forall \V \in \mD(G),
\quad \text{ in }  \mD'(0,T),
\end{equation}
with vanishing initial data. It follows from \cite{lions_non-homogeneous_1972} that this problem has a unique solution and it is zero.
\end{proof}

Existence and stability results with respect to the data $u_b$ are more difficult to obtain because of the essential inhomogeneous boundary condition. 
The common approach consists in decomposing the solution $\U$ as ${\U = \U_0 + L(u_b)}$, where the function $\U_0 \in \mD(G)$ is solution to a homogeneous problem, and the operator $L$ is a lifting operator.
We aim to define a lifting operator in a way that preserves the symmetry between the potential-based and the velocity-field problems. 
Hence the lifting should be defined as $L(u_b) = -G^* \Phi_b$, where $\Phi_b \in \mD(G_\alpha^*) $ is solution to the elliptic problem
\begin{equation} \label{eq:V:NonCoercive}
\forall \tilde \Phi = 
    \begin{pmatrix}
    \tilde \Phi
    \\
    \tilde {\bf V}
    \end{pmatrix}  \in \mD(G^*)
    \text{ with } \tilde \Phi = 
    \begin{pmatrix}
    \tilde \varphi \\ \tilde \psi \\ \tilde \gamma 
    \end{pmatrix},
    \quad 
     (G^* \Phi_b, G^* \tilde \Phi )_\mH 
    = \langle u_b, \, \traceB(\tilde \varphi)  
    \rangle_{\Gamma_b}. 
\end{equation}
However, in our case, defining such a lifting is not trivial because of the following result.

\begin{thrm} \label{prop:RangeG}
The range of the operators $G$ and $G^*$ are not closed. 
\end{thrm}
\begin{proof}
We consider only the case $d=2$, and show that there is no scalar $C>0$ such that, for all function $\U(t) \in \mD(G) \cap \text{Ker} (G)^\perp$, it holds
\begin{equation}\label{eq:proof:ineq}
\| G \U \|_\mG^2 
\geq C \| \U \|_\mH^2.
\end{equation}
This property implies that the range of $G$ is not closed, hence the range of $ G^*$ is also not closed. First note that $\text{Ker} (G) = \{ 0 \}$. Indeed, the kernel of $G$ is defined by
\[
\text{Ker} (G) = \{ \U \in \mD(G) \ | \ 
U_z = 0, \; 
\partial_x U_x = 0, \;
\normTraceS(\U) = 0 \},
\]
Hence, for every $\U = (U_x, U_z) \in \text{Ker} (G)$, it holds $U_z=0$ and $U_x$ is constant in the $x$ direction. Since the domain is infinite in the $x$ direction and $U_x \in L^2(\Omega)$, it is equal to zero. Hence $\mD(G^*) \cap \text{Ker} (G)^\perp = \mD(G^*)$. Now, let $ u \in \mD(\Omega)$ be a function with compact support in $ \R \times (H_-,H)$, 
and for each integer $n>1$ let the function $u_n \in \mD(\Omega)$ be defined by
$u_n ({\bf x}) = u(x,n (z-z_0))$
, with 
$z_0 = (H-H_-)/2$.
We have
\[
\| \partial_x u_n \|_{L^2(\Omega)}
= \frac{1}{n} \| \partial_x u \|_{L^2(\Omega)}, 
\quad 
\| \partial_z u_n \|_{L^2(\Omega)}
= \| \partial_z u \|_{L^2(\Omega)}.
\]
Then let $\U_n \in \mD(G)$ be the divergence-free function defined by
\[
\U_n(x,z) = \begin{pmatrix}
\partial_z u_n (x,z) \\ - \partial_x u_n(x,z)
\end{pmatrix}= \begin{pmatrix}
 n (\partial_z u)(x,n (z-z_0))  \\ - (\partial_x u)(x,n (z-z_0))
\end{pmatrix}.
\]
We have 
$G \U_n = (- g \partial_x u_n , N \partial_x u_n , 0)^t$ and 
\begin{equation}
\| G \U_n \|_\mG^2
= \frac{1}{n}
\left\| \sqrt{ \rho_0\frac{g^2 + N^2}{c_0^2}} \, \partial_x u \right\|_{L^2(\Omega)}^2,
\quad 
\| \U_n \|_\mH^2
= \frac{1}{n} 
  \| \sqrt{\rho_0}\partial_x u \|_{L^2(\Omega)}^2
+ \| \sqrt{\rho_0}  \partial_z u  \|_{L^2(\Omega)}^2.
\end{equation}
For $n$ going to infinity we have that $\| G \U_n \|_\mG^2  $  goes to zero and $ \| \U_n \|_\mH^2 $ converges to $  \| \sqrt{\rho_0} \partial_z u  \|_{L^2(\Omega)}^2 $, hence \eqref{eq:proof:ineq} can not hold.

\end{proof}

Since the range of $G$ is not closed, the bilinear form of the problem \eqref{eq:V:NonCoercive} is not coercive and the existence of a solution is not ensured. 
As a result, we cannot define in a straightforward way a lifting operator of the form $L(u_b) = -G^* \Phi_b$. 
Other lifting choices are possible but are not compatible with the potential-based formulation and are also not straightforward to analyze. To circumvent this problem, we introduce in what follows a dissipative version of the problem \eqref{eq:V:Weak_1}-\eqref{eq:V:Weak_CI_1} that will be easier to analyze. Existence and uniqueness results will be obtained using a limit process, namely, by letting the dissipation go to zero.


\subsection{Existence results} 

\subsubsection{A formulation with artificial dissipation} \label{sec:Velocity:Relaxed}
Instead of studying the existence of a solution for problem \eqref{eq:V:Weak_1}-\eqref{eq:V:Weak_CI_1}, we introduce a modified problem for a new unknown satisfying 
-- assuming that the  solution $\U(t)$ to \eqref{eq:V:Weak_1}-\eqref{eq:V:Weak_CI_1} exists -- 
for $ \alpha > 0 $  
\[
\U_\alpha(t) = e^{-\alpha t} \U(t),   \quad \text{ in }  \mD'(0,T). 
\]
The variational formulation for this new unknown is then: for $ u_b \in H^2(0,T;H^{-1/2}(\Gamma_b) ) $ given, find 
\begin{equation} \label{eq:V:Weak_2_Reg}
\U_\alpha \in L^2(0,T; \mD(\Ge)), \quad \frac{d}{dt}\U_\alpha \in L^2(0,T; \mH),
\end{equation}
solution to
\begin{align}
\frac{d^2}{dt^2} (\U_\alpha , \V )_\mH + 2 \alpha \frac{d}{dt} (\U_\alpha , \V )_\mH  
+ \alpha^2  (\U_\alpha , \V )_\mH 
+ (\Ge \U_\alpha, G \V)_\mG
= 0, 
& \quad \forall \V \in \mD(G),
\quad \text{ in }  \mD'(0,T),
\label{eq:V:Weak_2}
\\[4pt]
\normTraceB(\U_\alpha) = e^{-\alpha t} u_b, 
& \quad \text{ in }   (0,T),
\label{eq:V:Weak_BC_2}
\\
\U_\alpha(0) =  \frac{d}{dt} \U_\alpha(0) = 0.
\label{eq:V:Weak_CI_2}
\end{align}

The lemma below is straightforward to prove. 

\begin{lmm} Let 
$
\U_\alpha  $
be a solution to \eqref{eq:V:Weak_2}-\eqref{eq:V:Weak_CI_2}. Then $ e^{\alpha t} \U_\alpha $  is  the unique solution  to \eqref{eq:V:Weak_1}-\eqref{eq:V:Weak_CI_1}.   
\end{lmm}

As an immediate consequence of the lemma above, we have that problem \eqref{eq:V:Weak_2}-\eqref{eq:V:Weak_CI_2} has a unique solution. The relaxed formulation can be written in a more compact form by introducing the space
\[
     \mGAlpha = \mG \times \mH, \quad (\cdot,\cdot)_{\mGAlpha} = (\cdot,\cdot)_{\mG} + (\cdot,\cdot)_{\mH},
\]
and the operator $G_\alpha: \mD(G_\alpha) \subset \mH  \to \mGAlpha$, defined by $\mD(G_\alpha) = \mD(G)$ and
\[
  \forall \, \U \in \mD(G_\alpha), \; \,
G_\alpha  \U 
= \begin{pmatrix}
G \U
\\[4pt]
\alpha \U
\end{pmatrix}.
\]
The operator 
$\Ge_\alpha: \mD(\Ge_\alpha) \subset \mH  \to \mGAlpha$ 
is defined similarly, and it holds $\mD(\Ge_\alpha) = \mD(\Ge)$. The variational formulation \eqref{eq:V:Weak_2}-\eqref{eq:V:Weak_CI_2} is then equivalent to: 
find $\U_\alpha$ with the regularity \eqref{eq:V:Weak_2_Reg} and solution to
\begin{align}
\frac{d^2}{dt^2} (\U_\alpha , \V )_\mH + 2 \alpha \frac{d}{dt} (\U_\alpha , \V )_\mH   
+ (\Ge_\alpha \U_\alpha, G_\alpha \V)_\mGAlpha
= 0, 
& \quad \forall \V \in \mD(G_\alpha),
\quad \text{ in }  \mD'(0,T),
\label{eq:V:Weak_3}
\\[4pt]
\normTraceB(\U_\alpha) = e^{-\alpha t} u_b, 
& \quad \text{ in } (0,T),
\label{eq:V:Weak_BC_3}
\\[4pt]
\U_\alpha(0) =  \frac{d}{dt} \U_\alpha(0) = 0.
\label{eq:V:Weak_CI_3}
\end{align}


\subsubsection{Lifting operator for the dissipative problem}
The operators $ G_\alpha$ and  $ \Ge_\alpha$ are densely defined and closed, so are their adjoints $G_\alpha^\ast$ and  $ \Ge_\alpha^\ast $. We have, in particular, 
\begin{equation}\label{def_G_ast_alpha}
\mD(G_\alpha^*) = \mD(G^\ast) \times \mH, \quad 
\text{ and } \quad
\forall \, (\Phi, {\bf V}) \in  \mD(G_\alpha^*), \quad
G_\alpha^\ast  
\begin{pmatrix}
\Phi \\ {\bf V}
\end{pmatrix}
=  G^\ast \Phi + \alpha {\bf V} \in \mH.
\end{equation}
 Moreover $\mD(\Ge^{\, *}_\alpha) = \mD(\Ge^{\, *})  \times \mH $ and $ G_\alpha^* $  is an extension of  $\Ge^{\, *}_\alpha.$
In the following the identity operators in $\mH$ and in $\mGAlpha$ will be used. 
By abuse of notation, both are denoted $\text{I}$.
 
\begin{thrm}
The range of the operators $G_\alpha$ and $G_\alpha^*$ are closed. 
\end{thrm}

\begin{proof}
Since $G_\alpha = \begin{pmatrix}
G \\ \alpha \, \text{I}
\end{pmatrix}$, we have $\text{Ker } G_\alpha = \emptyset$. 
Then, for $\U \in \mD(G_\alpha) \cap \text{Ker}(G_\alpha)^\perp = \mD(G_\alpha)$, it holds
\[
\| G_\alpha \U \|_\mGAlpha^2
= \| G \U \|_\mG^2 + \| \alpha \U \|_\mH^2 \geq \alpha^2 \| \U \|_\mH^2,
\]
which concludes the proof. 
\end{proof}

This result is key for constructing a lifting operator. We recall that the kernel of the operator $ G^*_\alpha $ is a closed subspace of
$\mGAlpha$ and we introduce $Q_\alpha \in \mathcal{L}(\mGAlpha)$, the orthogonal projection on
$\mbox{Ker } (G^*_\alpha)$. 
Then $Q_\alpha^2 = Q_\alpha $ and $Q_\alpha$ is self-adjoint.

\begin{thrm}\label{Poincare_abstract}
For $\alpha \leq 1$ we have
\[
\forall \, \PhiAlpha = \begin{pmatrix}
\Phi \\ {\bf V}
\end{pmatrix}\in \mD(G_\alpha^*),
\quad
\| Q_\alpha \PhiAlpha \|_{\mGAlpha}^2 
  + \| G_\alpha^* \PhiAlpha \|_{\mH}^2 
  \geq \alpha^2 \| \PhiAlpha \|_{\mGAlpha}^2 .
\] 
\end{thrm}

\begin{proof}
Let $\PhiAlpha \in \mD(G_\alpha^*)$. We have 
$
  \| Q_\alpha \PhiAlpha \|_{\mGAlpha}^2
+ \| (\text{I}-Q_\alpha) \PhiAlpha \|_{\mGAlpha}^2
= \| \PhiAlpha \|_{\mGAlpha}^2,
$
and  therefore
\[
\| G_\alpha^* \PhiAlpha \|_{\mH}^2
= \| G_\alpha^* (\text{I}-Q_\alpha) \PhiAlpha \|_{\mH}^2
\geq  \alpha^2  \, \| (\text{I}-Q_\alpha) \PhiAlpha \|_{\mGAlpha}^2.
\]
Hence we have
\[
\| Q_\alpha \PhiAlpha \|_{\mGAlpha}^2 
+ \| G_\alpha^* \PhiAlpha \|_{\mH}^2 
\geq   
  \| Q_\alpha \PhiAlpha \|_{\mGAlpha}^2 
+ \alpha^2 \, \| (\text{I}-Q_\alpha) \PhiAlpha \|_{\mGAlpha}^2 
\geq \alpha^2 \| \PhiAlpha \|_{\mGAlpha}^2.
\] 
\end{proof}

The exact expression of $Q_\alpha$ is of no practical interest in what follows. It has to be noted however that it is non trivial, due to the fact that $\text{Ker } G_\alpha^*$ is an infinite dimensional space.
As a direct consequence of Proposition \ref{prop:GraphNorm}, we have  the following bound for the $H^1(\Omega)$-norm, 
\begin{lmm}\label{prop:Lift_Continuity}
Let $\alpha < 1$, there exists a scalar $C_c > 0$ such that
\[
\forall \, \PhiAlpha = 
\begin{pmatrix}
\Phi \\ {\bf V}
\end{pmatrix} \in \mD(G_\alpha^*)
\text{ with } 
\Phi = \begin{pmatrix}
\varphi \\ \psi \\ \gamma
\end{pmatrix}, 
\quad
\| \PhiAlpha \|_{\mGAlpha}^2
+ \| G_\alpha^* \PhiAlpha \|_{\mH}^2
\geq C_c  \| \varphi \|_{H^1(\Omega)}^2 .
\]
\end{lmm}

\begin{proof} By definition and the triangle inequality,
\[
\| \PhiAlpha \|_{\mGAlpha}^2
+ \| G_\alpha^* \PhiAlpha \|_{\mH}^2 \geq \| \Phi \|_{\mG}^2 +  \| {\bf V} \|_{\mH}^2
+ \big(\| G^* \Phi \|_{\mH} - \alpha \| {\bf V} \|_{\mH} \big)^2.
\]
The conclusion follows from the Young's inequality, the assumption on $\alpha$ and Proposition \ref{prop:GraphNorm}.
\end{proof}

Aiming to define the aforementioned lifting operator of a data $u_b  \in H^{-1/2}(\Gamma_b)$ (the dependence in time is omitted), we introduce the following problem: for all $u_b$, 
find $ \PhiAlpha_b \in \mD(G_\alpha^*) $ solution to
\begin{equation} \label{eq:V:Lift}
\forall \tilde \PhiAlpha = 
    \begin{pmatrix}
    \tilde \Phi
    \\
    \tilde {\bf V}
    \end{pmatrix}  \in \mD(G^*_\alpha)
    \text{ with } \tilde \Phi = 
    \begin{pmatrix}
    \tilde \varphi \\ \tilde \psi \\ \tilde \gamma 
    \end{pmatrix},
    \quad 
    (Q_\alpha  \PhiAlpha_b ,  \tilde \PhiAlpha )_\mGAlpha 
    + (G_\alpha^* \PhiAlpha_b, G_\alpha^* \tilde \PhiAlpha )_\mH 
    = \langle u_b, \, \traceB(\tilde \varphi)  
    \rangle_{\Gamma_b}. 
\end{equation}
Thanks to Theorem \ref{Poincare_abstract} and Lemma \ref{prop:Lift_Continuity}, it is classical to show that 
Equation \eqref{eq:V:Lift} admits a unique solution 
$\PhiAlpha_b \in  \mD(G_\alpha^*)$ that depends continuously on $u_b$ (it is a standard application of Lax-Milgram Lemma). 
We then construct the lifting operator $L_\alpha \in \mathcal L(  H^{-1/2}(\Gamma_b), \mH)$ by setting
\[
\forall u_b \in  H^{-1/2}(\Gamma_b), \quad 
L_\alpha (u_b) 
=  - G_\alpha^* \PhiAlpha_b.
\]

\begin{prpstn}\label{property_L}
The function $ L_\alpha(u_b) \in \mH $ has the following properties:
  \[L_\alpha(u_b) \in \mD(\Ge_\alpha)= \mD(\Ge), 
  \quad \Ge_\alpha L_\alpha(u_b) = Q_\alpha \PhiAlpha_b \in \mbox{\normalfont Ker }(G^*_\alpha),
  \quad  \mbox{and} \quad \normTraceB(L_\alpha(u_b))  = u_b.
  \]
\end{prpstn}

\begin{proof}
In the equation \eqref{eq:V:Lift} we choose a  test function $\tilde \PhiAlpha$ in the space $\mD(\Ge^{\, *}_\alpha)$. Since $\gamma_{0,b}(\tilde \varphi) = 0 $,  we obtain
\begin{equation}
(G_\alpha^* \PhiAlpha_b, \Ge^{\, *}_\alpha \tilde \PhiAlpha)_\mH
= -(Q_\alpha \PhiAlpha_b , \tilde \PhiAlpha)_{\mGAlpha} , 
\quad \forall \tilde \PhiAlpha \in \mD(\Ge^{\, *}_\alpha).
\end{equation}
This implies 
\begin{equation}\label{eq:GGphi}
G_\alpha^* \PhiAlpha_b \in \mD(\Ge^{\, **}_\alpha)
= \mD(\Ge_\alpha)  = \mD(\Ge) 
\quad \mbox{ and } \quad  
\Ge^{\, **}_\alpha G_\alpha^* \PhiAlpha_b 
= \Ge_\alpha G_\alpha^* \PhiAlpha_b = - Q_\alpha \PhiAlpha_b =  -  \Ge_\alpha L_\alpha(u_b) ,
\end{equation}
where we have used that $\Ge_\alpha^{**} = \Ge_\alpha$ since $ \Ge_\alpha $ is closed and densely defined. The first two properties of the proposition are proved. 
To prove the last property, we use the abstract Green's formula of Lemma \ref{GreenFormula}: for all $\widetilde \PhiAlpha \in \mD(G^*_\alpha)$ we have,
\begin{align}
(G_\alpha^* \PhiAlpha_b, G_\alpha^* \widetilde \PhiAlpha)_\mH
& = (G_\alpha^* \PhiAlpha_b, G^* \widetilde \Phi)_\mH
+ (G_\alpha^* \PhiAlpha_b, \alpha \widetilde {\bf V})_\mH 
\\
& =   (\Ge (G_\alpha^* \PhiAlpha_b), \widetilde \Phi)_\mH 
- \langle \normTraceB(G_\alpha^* \PhiAlpha_b), \, 
\traceB(\tilde \varphi)  
  \rangle_{\Gamma_b}
+ (G_\alpha^* \PhiAlpha_b, \alpha \widetilde {\bf V})_\mH.
\end{align}
Hence we have
\[
(G_\alpha^* \PhiAlpha_b, G_\alpha^* \widetilde \PhiAlpha)_\mH 
= (\Ge_\alpha (G_\alpha^* \PhiAlpha_b), \widetilde \PhiAlpha)_\mGAlpha
- \langle \normTraceB(G_\alpha^* \PhiAlpha_b), \, 
\traceB(\tilde \varphi)  
  \rangle_{\Gamma_b},
\]
and with the equations \eqref{eq:V:Lift} and \eqref{eq:GGphi} we obtain
\begin{align}
(G_\alpha^* \PhiAlpha_b, G_\alpha^* \widetilde \PhiAlpha)_\mH 
& = - ( Q_\alpha \PhiAlpha_b, \widetilde \PhiAlpha)_\mGAlpha
- \langle \normTraceB(G_\alpha^* \PhiAlpha_b), \, 
\traceB(\tilde \varphi)  
  \rangle_{\Gamma_b}
\\
& = - ( Q_\alpha \PhiAlpha_b, \widetilde \PhiAlpha)_\mGAlpha
+ \langle u_b, \, \traceB(\tilde \varphi)  
  \rangle_{\Gamma_b},
\end{align}
from which we deduce that 
$ \normTraceB (L_\alpha(u_b))  = - \normTraceB(G_\alpha^* \PhiAlpha_b) = u_b.$  
\end{proof}

\begin{rmrk} \label{remark:Phi_b}
Assume now that $  u_b \in L^2( 0,T; H^{-1/2}(\Gamma_b) ) $. 
The operator $L_\alpha$ being continuous from $ H^{-1/2}(\Gamma_b) $ to $ \mH $, one can deduce that
\[
   L_\alpha(u_b) \in L^2(0,T; \mH).
\]
Proposition \ref{property_L} can be extended: we have, for almost all $t \in (0,T)$,
\[
  L_\alpha(u_b(t)) \in \mD(\Ge_\alpha),  \quad 
  \Ge_\alpha L_\alpha(u_b(t)) \in \mbox{\normalfont Ker } G_\alpha^*, 
  \quad \text{and} \quad 
  \normTraceS(L_\alpha(u_b(t)))  = u_b(t).
\]
Moreover, the regularity in time of $ L_\alpha(u_b)$ depends straightforwardly on the regularity in time of $u_b$. In particular,
\begin{equation}
    u_b \in H^{k}(0,T ; H^{-1/2}(\Gamma_b))
    \ \Rightarrow \
    L_\alpha(u_b) \in H^{k}(0,T; \mD(\Ge)).
\end{equation}
\end{rmrk}


\subsubsection{Existence result} \label{sec:V:Existence}
We show now the existence of solution to the dissipative problem.

\begin{prpstn} \label{prop:V:Existence}
If $u_b \in H^2(0,T; H^{-1/2}(\Gamma_b))$ and $u_b(0) = \frac{d}{dt} u_b(0) = 0,$
the problem \eqref{eq:V:Weak_3}-\eqref{eq:V:Weak_CI_3} admits a unique solution $\U_\alpha$. It satisfies, up to modifications on zero-measure sets,
\begin{equation}\label{reg_U_dyn}
     \U_\alpha \in C^1([0,T]; \mH) \cap C^0([0,T] ; \mD(\Ge)).
\end{equation}
\end{prpstn}

\begin{proof}
Using the lifting operator, we look first for ${\U_{\alpha,0}(t) \in \mD(G)}$ such that, for all $ \V \in \mD(G)$,
\begin{align}
& \nonumber \frac{d^2}{dt^2} (\U_{\alpha,0}, \V )_\mH 
+ 2 \alpha \frac{d}{dt} (\U_{\alpha,0}, \V )_\mH   
+ (G_\alpha \U_{\alpha,0}, G_\alpha \V)_\mGAlpha \\ & \nonumber \hspace{45pt}  
 = \left( \frac{d^2}{dt^2} \big(e^{-\alpha t} L_\alpha 
(    u_b ) \big)
+ 2 \alpha  \frac{d}{dt} \big( e^{-\alpha t}  L_\alpha 
(    u_b )  \big), \V \right)_\mH
\\ & \hspace{45pt}  
=  e^{-\alpha t} \left(   \frac{d^2}{dt^2} L_\alpha (u_b)  - \alpha^2 L_\alpha (u_b),  \V \right)_\mH
\qquad  \qquad  \qquad \text{in }  \mD'(0,T), \label{eq:V:Existence}
\end{align} 
with vanishing initial conditions.
From the remark above, the data is in 
$L^2(0,T; \mD(\Ge))$.
The existence and uniqueness of $\U_{\alpha,0}$ is then obtained by application of  standard results \cite{dautray_mathematical_2000}, and $\U_{\alpha,0}$ has the following regularity,
\begin{equation}
    \U_{\alpha,0} \in C^1([0,T]; \mH) \cap C^0([0,T] ; \mD(G)).
\end{equation}
Let $\U_\alpha = \U_{\alpha,0} + L_\alpha(e^{-\alpha t} u_b)$. 
Since we have --  up to modifications on zero-measure sets --  that the inclusion of  $H^2(0,T; \mH)  $ into $ C^1([0,T]; \mH)$ (see \cite{evans_partial_2004}),  we deduce that
\begin{equation}
    \U_\alpha \in C^1([0,T]; \mH) \cap C^0([0,T] ; \mD(\Ge)).
\end{equation}
Finally we show that $\U_\alpha$ is solution to \eqref{eq:V:Weak_3}. 
The computations are mostly straightforward, but we give some details for the following term,
\begin{equation}\label{eq:manip_alpha}
(\Ge_\alpha \U_{\alpha}, G_\alpha \V)_\mGAlpha
= (\Ge_\alpha \U_{\alpha, 0}, G_\alpha \V)_\mGAlpha
+ (\Ge_\alpha L_\alpha (e^{-\alpha t}), G_\alpha \V)_\mGAlpha
\end{equation}
Since $ G_\alpha \subset \Ge_\alpha$, it holds $\Ge_\alpha \U_{\alpha, 0} = G_\alpha \U_{\alpha, 0}$, hence the first term of \eqref{eq:manip_alpha} is replaced using \eqref{eq:V:Existence}.
For the second term of \eqref{eq:manip_alpha}, from the properties of the lifting operator given in Proposition \ref{property_L}, it holds
\begin{equation}
(\Ge_\alpha L_\alpha (e^{-\alpha t}u_b), G_\alpha \V)_\mGAlpha
=  e^{-\alpha t} (Q_\alpha \PhiAlpha_b , G_\alpha \V)_\mGAlpha = e^{-\alpha t} (G_\alpha^* Q_\alpha \PhiAlpha_b , \V)_\mH
= 0, 
\end{equation}
which concludes the proof. 
\end{proof}




\section{From potential-based solutions to velocity-field solutions} \label{sec:Equivalence}
In Sections \ref{sec:Potential} and \ref{sec:Velocity}, we have proved that both the velocity-field problem \eqref{eq:V:Weak_1}-\eqref{eq:V:Weak_CI_1} and the potential-based problem  \eqref{eq:P:Standard_weak}-\eqref{eq:P:Standard_CI} have a unique solution.
The aim of this section is to prove that velocity-fields solution to \eqref{eq:V:Weak_1}-\eqref{eq:V:Weak_CI_1} can be constructed from potential-based solutions. More precisely, we will prove the following theorem,

\begin{thrm} \label{th:Equivalence}
Let   $u_b \in H^2(0,T;H^{-1/2}(\Gamma_b))$ and $u_b(0) = \frac{d}{dt} u_b(0) = 0$, and let $\Phi$ be the solution to the problem \eqref{eq:P:Standard_weak}-\eqref{eq:P:Standard_CI} with source term $u_b$. 
Then $\U = G^* \Phi$ is the unique solution to the problem \eqref{eq:V:Weak_1}-\eqref{eq:V:Weak_CI_1} with the same source term.
\end{thrm}

A key ingredient in the following analysis is the Von Neumann theorem 
\cite{grubb_distributions_2009}, that we recall below for the sake of completeness. 
\begin{thrm}[Von Neumann]
If $T:\mD(T) : \mH \to \mG$ is a closed densely defined operator, then 
$T^* T$ is self-adjoint and $\mD(T^* T)$ is dense in $\mD(T)$. 
\end{thrm}

As already shown in Sec. \ref{sec:V:Uniqueness}, the solution to the problem \eqref{eq:V:Weak_1}-\eqref{eq:V:Weak_CI_1} should be defined as the sum of the solution to a homogeneous problem and the lifting of the source term. 
As presented in Section \ref{sec:Velocity:Relaxed}, the definition of the lifting requires in turn to consider the dissipative problem \eqref{eq:V:Weak_3}-\eqref{eq:V:Weak_CI_3}. 
For this reason, it is more convenient to prove first the equivalence between the dissipative problem  \eqref{eq:V:Weak_3}-\eqref{eq:V:Weak_CI_3} and the corresponding potential-based formulation, that we introduce in the following section.
The result stated in Theorem \ref{th:Equivalence} for the non-dissipative problem \eqref{eq:P:Standard_weak}-\eqref{eq:P:Standard_CI} and \eqref{eq:V:Weak_1}-\eqref{eq:V:Weak_CI_1} is then deduced through a limit process carried out at the end of this section.


\subsection{The case with artificial dissipation}
In this section we recall the dissipative problem \eqref{eq:V:Weak_3}, and define the associated ``dual'' or ``adjoint'' problem.
For simplicity, we first consider volume sources, and show then how to deduce a result similar to Theorem \ref{th:Equivalence} for problems with boundary sources. 
The dissipative formulation with a volume source is: let $F_\U$ regular enough, find
\begin{equation}
\U_\alpha \in L^2(0,T; \mD(\Ge_\alpha)), \quad \frac{d}{dt}\U_\alpha \in L^2(0,T; \mH),
\end{equation}
 solution to
\begin{align}
\frac{d^2}{dt^2} (\U_\alpha, \V)_\mH + 2 \alpha \frac{d}{dt} (\U_\alpha, \V)_\mH
+ (G_\alpha \U_\alpha, G_\alpha \V)_\mGAlpha
= (F_\U, \V)_\mH, 
& \quad  \forall \V \in \mD(G_\alpha), \quad 
\text{in } \mD'(0,T),
\label{eq:Equiv:U_Weak}
\\
\U_\alpha(0) = \frac{d}{dt}\U_\alpha(0) = 0.
\label{eq:Equiv:U_CI}
\end{align}
The potential-based formulation is defined using the adjoint $G_\alpha^*$, and reads: given $F_\Phi$ regular enough, find
\[
\PhiAlpha_{\alpha} \in L^2(0,T;\mD(G_\alpha^*)), \quad   \frac{d}{dt} \PhiAlpha_{\alpha} \in  L^2(0,T;\mGAlpha)
\] solution to: for all $\tilde \PhiAlpha \in \mD(G_\alpha^*)$
\begin{align}
\frac{d^2}{dt^2} (\PhiAlpha_\alpha , \tilde \PhiAlpha )_\mGAlpha
+ 2 \alpha \frac{d}{dt} 
(\PhiAlpha_\alpha , \tilde \PhiAlpha)_\mGAlpha   
+ (G_\alpha^* \PhiAlpha_\alpha , G_\alpha^* \tilde \PhiAlpha)_\mH 
= (F_\Phi, \tilde \PhiAlpha)_\mGAlpha, 
&  
\quad \mD'(0,T) ,
\label{eq:Equiv:Phi_Weak}
\\
\PhiAlpha_\alpha (0) =\frac{d}{dt} \PhiAlpha_\alpha (0) = 0.
\label{eq:Equiv:Phi_CI}
\end{align}
The existence and uniqueness of the solution to the problems \eqref{eq:Equiv:U_Weak}-\eqref{eq:Equiv:U_CI} and \eqref{eq:Equiv:Phi_Weak}-\eqref{eq:Equiv:Phi_CI} are a direct application of well-known results, given e.g. in \cite{dautray_mathematical_2000}. Whenever $F_{\U} \in L^2(0,T; \mH) $ and $F_{\U} \in L^2(0,T; \mGAlpha), $
there exists 
a unique solution $\U_\alpha$ to \eqref{eq:Equiv:U_Weak}-\eqref{eq:Equiv:U_CI}, and 
a unique solution $\PhiAlpha_\alpha $ to \eqref{eq:Equiv:Phi_Weak}-\eqref{eq:Equiv:Phi_CI}, and they satisfy
\[
 \U_\alpha \in C^1([0,T]; \mH)
\cap C^0([0,T]; \mD(G)), \quad
 \PhiAlpha_\alpha \in C^1([0,T]; \mGAlpha)
\cap C^0([0,T]; \mD(G_\alpha^*)).
\]
We state now a relation property between  problems \eqref{eq:Equiv:U_Weak}-\eqref{eq:Equiv:U_CI} and  \eqref{eq:Equiv:Phi_Weak}-\eqref{eq:Equiv:Phi_CI}.

\begin{thrm} 
\label{th:Equivalence_Volume}
Let ${  \PhiAlpha_\alpha \in \mD(G^*_\alpha)}$ be the solution to the problem 
\eqref{eq:Equiv:Phi_Weak}-\eqref{eq:Equiv:Phi_CI}.
If ${F_\Phi \in L^2(0,T;\mD(G^*_\alpha))}$, then ${\U_\alpha = G^*_\alpha \PhiAlpha_\alpha }$ is the unique solution to the problem 
\eqref{eq:Equiv:U_Weak}-\eqref{eq:Equiv:U_CI}, with source term defined by ${F_\U = G_\alpha^* F_\Phi}$.
\end{thrm}

\begin{proof} 
We introduce the function $ {\boldsymbol \Psi}_\alpha$ defined by
\[
     {\boldsymbol \Psi}_\alpha(t) = \int_0^t \PhiAlpha_\alpha(s) \xdif s.
\]
Integrating in time \eqref{eq:Equiv:Phi_Weak} yields an equation for  $ {\boldsymbol \Psi}_\alpha $, for all $ \tilde \PhiAlpha \in \mD( G_\alpha^*), $ 
\begin{equation}\label{eq:proof:eq_for_psi}
\frac{d^2}{dt^2} ( {\boldsymbol \Psi}_\alpha  , \tilde \PhiAlpha )_\mGAlpha
+ 2 \alpha \frac{d}{dt} 
( {\boldsymbol \Psi}_\alpha  , \tilde \PhiAlpha)_\mGAlpha   
+ (G_\alpha^*  {\boldsymbol \Psi}_\alpha  , G_\alpha^* \tilde \PhiAlpha)_\mH 
= \big( \int_0^t F_\Phi(s) \xdif s, \tilde \Phi \big)_\mGAlpha , 
\end{equation}
and since we have by construction $ {\boldsymbol \Psi}_\alpha \in C^2([0,T]; \mGAlpha)
\cap C^1([0,T]; \mD(G^*_\alpha)) $, we deduce from \eqref{eq:proof:eq_for_psi} that
\begin{equation} \label{eq:Equivalence:Regularity}
 {\boldsymbol \Psi}_\alpha \in C^0([0,T];\mD(G_\alpha G^*_\alpha)).
\end{equation}
Now, for all $\V \in D(G^\ast_\alpha G_\alpha)$ we set $ \tilde \PhiAlpha  = G_\alpha  \V $ in \eqref{eq:proof:eq_for_psi}. We obtain, using the definition of the adjoint of $G_\alpha$, Equation \eqref{eq:Equivalence:Regularity} and the assumption that ${F_\Phi \in L^2(0,T;\mD(G^*))}$, 
\begin{equation}
\frac{d^2}{dt^2} (G_\alpha^*  {\boldsymbol \Psi}_\alpha  ,\V  )_\mH
+ 2 \alpha \frac{d}{dt} 
( G_\alpha^*  {\boldsymbol \Psi}_\alpha  , \V )_\mH   
+ ( G_\alpha G_\alpha^*  {\boldsymbol \Psi}_\alpha  ,  G_\alpha  \V)_\mGAlpha 
= \big( \int_0^t G_\alpha^*  F_\Phi(t') \xdif t', \V \big)_\mH.
\end{equation}
Therefore, setting ${\bf V}_\alpha = G_\alpha^*{\boldsymbol \Psi}_\alpha$, we have that 
\[
{\bf V}_\alpha \in C^1([0,T]; \mH) \cap C^0([0,T]; \mD(G)) 
\]
is solution to: for all $\V \in \mD(G^\ast_\alpha G_\alpha)$,
\begin{equation}\label{eq:proof:eq_for_V}
\frac{d^2}{dt^2} ( {\bf V}_\alpha   ,\V  )_\mH
+ 2 \alpha \frac{d}{dt} 
( {\bf V}_\alpha , \V )_\mH   
+ ( G_\alpha{\bf V}_\alpha  ,  G_\alpha  \V)_\mH 
= \big( \int_0^t G_\alpha^*  F_\Phi(t')\xdif t', \V \big)_\mH.
\end{equation}
 Since $G_\alpha$ is a closed densely defined operator, $\mD(G^*_\alpha G_\alpha)$ is dense in $\mD(G_\alpha)$ by the Von Neumann theorem, hence \eqref{eq:proof:eq_for_V} can be extended to functions in $\V \in \mD(G_\alpha) = \mD(G) $. At this point we have shown that 
\[
    {\bf V}_\alpha =  \int_0^t  G_\alpha^*  \PhiAlpha_\alpha(s) \xdif s
\]
is solution to \eqref{eq:Equiv:U_Weak}-\eqref{eq:Equiv:U_CI}, with source term given by 
\[
 F_\U(t)  =  \int_0^t G_\alpha^*  F_\Phi(s)  \xdif s.
\]
Since $  F_\U  $ is differentiable in time  by construction, one can deduce  that the time derivative of $ {\bf V}_\alpha  $ is solution to $ \eqref{eq:Equiv:U_Weak}-\eqref{eq:Equiv:U_CI} $ with source term $ G_\alpha^*  F_\Phi $, which allows us to conclude.
\end{proof}

We consider now the case of a source term located at the bottom, namely the dissipative problem \eqref{eq:V:Weak_3}-\eqref{eq:V:Weak_BC_3} and its potential-based formulation: find
\[
 \PhiAlpha_{\alpha} \in L^2(0,T;\mD(G_\alpha^*)), \quad   \frac{d}{dt} \PhiAlpha_{\alpha} \in  L^2(0,T;\mGAlpha)
 \] solution to
\begin{align}
\frac{d^2}{dt^2} (\PhiAlpha_\alpha , \tilde \PhiAlpha )_\mGAlpha
+ 2 \alpha \frac{d}{dt} 
(\PhiAlpha_\alpha , \tilde \PhiAlpha)_\mGAlpha   
+ (G_\alpha^* \PhiAlpha_\alpha , G_\alpha^* \tilde \PhiAlpha)_\mH 
= \langle e^{-\alpha t} u_b , \traceB(\tilde \varphi) \rangle_{\Gamma_b},
& \quad \forall \tilde \PhiAlpha \in \mD(G_\alpha^*), 
\quad \text{ in }  \mD'(0,T),
\label{eq:P_Alpha:Weak}
\\[8pt]
\PhiAlpha_\alpha(0) = \frac{d}{dt} \PhiAlpha_\alpha(0) = 0.
\label{eq:P_Alpha:Weak_CI}
\end{align} 
As soon as the right-hand side of \eqref{eq:P_Alpha:Weak} defines a functional in  $H^2(0,T;\mD(G^*_\alpha)')$ (see the proof of Proposition \ref{prop:P:Existence}), the solution to \eqref{eq:P_Alpha:Weak}-\eqref{eq:P_Alpha:Weak_CI} exists and is unique. 
The result below is deduced from Theorem \ref{th:Equivalence_Volume} above and the property of the lifting operator given in Proposition \ref{property_L}.

\begin{prpstn}
Let  $u_b \in H^2(0,T;H^{-1/2}(\Gamma_b))$ and $u_b(0) = \frac{d}{dt} u_b(0) = 0$, and let $\PhiAlpha_{\alpha}$ be the solution to the problem \eqref{eq:P_Alpha:Weak}-\eqref{eq:P_Alpha:Weak_CI}.
Then $\U_\alpha = G_\alpha^* \PhiAlpha_{\alpha}$ is the unique solution to the problem \eqref{eq:V:Weak_3}-\eqref{eq:V:Weak_CI_3}.
\end{prpstn}

\begin{proof}
Let $ \PhiAlpha_b(t)$ be the solution to \eqref{eq:V:Lift}
associated to the data $ u_b(t)$ for almost all time $t\in [0,T]$. We have $ \PhiAlpha_b \in H^2(0,T;\mD(G^*_\alpha)) $ and the equation \eqref{eq:P_Alpha:Weak} can be written
\begin{equation}
\frac{d^2}{dt^2} (\PhiAlpha_{\alpha} , \tilde \PhiAlpha )_\mGAlpha   
+ 2 \alpha \frac{d}{dt} 
(\PhiAlpha_{\alpha}  , \tilde \PhiAlpha)_\mGAlpha   
+ (G_\alpha^* \PhiAlpha_{\alpha}  , G_\alpha^* \tilde \PhiAlpha)_\mH
= e^{-\alpha t} (Q_\alpha \PhiAlpha_b, \tilde \PhiAlpha)_\mGAlpha
+ e^{-\alpha t} (G_\alpha^* \PhiAlpha_b, G_\alpha^* \tilde \PhiAlpha)_\mH.
\end{equation}
Let $\PhiAlpha_0 = \PhiAlpha_{\alpha} - e^{-\alpha t} \PhiAlpha_b$ and
 define 
\begin{equation*}
\begin{array}{ll}
F_\PhiAlpha
& \!  \displaystyle = 
- \frac{d^2}{dt^2}  ( e^{-\alpha t} \PhiAlpha_b ) 
- 2 \alpha \frac{d}{dt}   ( e^{-\alpha t} \PhiAlpha_b )
+ e^{-\alpha t} Q_\alpha \PhiAlpha_b \\[8pt]
&  \!  \displaystyle = e^{-\alpha t} \big(\alpha^2 \PhiAlpha_b - \frac{d^2}{dt^2} \PhiAlpha_b +  Q_\alpha \PhiAlpha_b   \big)
\in L^2(0,T;\mD(G^*_\alpha)),
\end{array}
\end{equation*}
then $\PhiAlpha_0$ is the unique solution to: for all $\tilde \PhiAlpha \in \mD(G^*_\alpha), $
\[
\frac{d^2}{dt^2} (\PhiAlpha_0 , \tilde \PhiAlpha )_\mGAlpha
+ 2 \alpha \frac{d}{dt} 
(\PhiAlpha_0 , \tilde \PhiAlpha)_\mGAlpha   
+ (G^*_\alpha \PhiAlpha_0, G^*_\alpha \tilde \PhiAlpha)_\mH
=
( F_\PhiAlpha , \tilde \PhiAlpha )_\mGAlpha,
\]
with vanishing initial data (since $u_b$ and $\frac{d}{dt} u_b$ vanish at the initial time).
Using Theorem \ref{th:Equivalence_Volume}, the function $\U_{\alpha,0}$ defined by  $\U_{\alpha, 0} = G_\alpha^* \PhiAlpha_0 \in \mD(G_\alpha)$ is solution to 
\[
\frac{d^2}{dt^2} (\U_{\alpha,0}, \V)_\mH
+ 2 \alpha \frac{d}{dt} (\U_{\alpha,0} , \V)_\mH
+ (G_\alpha \U_{\alpha,0}, G_\alpha \V)_\mGAlpha
=
( G^*_\alpha F_\PhiAlpha , \V)_\mGAlpha,
\]
where, by construction,
\[
   G^*_\alpha F_\PhiAlpha =  e^{-\alpha t} \big(  
    \frac{d^2}{dt^2} L_\alpha(u_b)  - \alpha^2 L_\alpha(u_b) 
    \big).
\]
This shows that $\U_{\alpha, 0} = G_\alpha^* \PhiAlpha_0$ is indeed the solution to \eqref{eq:V:Existence}. It remains only to define, as in the proof of Theorem \ref{prop:V:Existence}, the function 
$\U_\alpha = \U_{\alpha,0} + e^{-\alpha t}  L_\alpha(u_b) 
= \U_{\alpha, 0} - e^{-\alpha t}  G^*_\alpha \PhiAlpha_b$ 
and observe that it is solution to \eqref{eq:V:Weak_3} with the required regularity. 
\end{proof}

We have shown that for problems with artificial dissipation, the solution $\U_\alpha$ to the velocity-field problem \eqref{eq:V:Weak_3}-\eqref{eq:V:Weak_CI_3} is obtained from the solution $\PhiAlpha_\alpha$ to the potential-based problem \eqref{eq:P_Alpha:Weak}-\eqref{eq:P_Alpha:Weak_CI}, thanks to the relation $\U_\alpha = G_\alpha^* \PhiAlpha_\alpha$. 
The next step is to prove a similar relation for the dissipative free problems \eqref{eq:V:Weak_1}-\eqref{eq:V:Weak_3} and \eqref{eq:P:Standard_weak}-\eqref{eq:P:Standard_CI}. 
It should be noted that the relation between the dissipative potential-based problem \eqref{eq:P_Alpha:Weak}-\eqref{eq:P_Alpha:Weak_CI} and its non dissipative version \eqref{eq:P:Standard_weak}-\eqref{eq:P:Standard_CI} are not related as obviously as for the velocity-field problems. In particular, one does not have a relation of the form $\Phi(t) = e^{\alpha t} \PhiAlpha_\alpha(t)$, where  $\Phi$ is the solution to \eqref{eq:P:Standard_weak}-\eqref{eq:P:Standard_CI} and $\PhiAlpha_\alpha$ is the solution to \eqref{eq:P_Alpha:Weak}-\eqref{eq:P_Alpha:Weak_CI}. 
Deducing $\Phi$ from $\PhiAlpha_\alpha$ requires to study the  limit process   $\alpha \to 0$.



\subsection{The dissipation free case}
At this point we have shown that $\U$, the solution to \eqref{eq:V:Weak_1}-\eqref{eq:V:Weak_CI_1}, can be expressed as
\[
    \U = e^{\alpha t} \U_\alpha = e^{\alpha t}  G_\alpha^* \PhiAlpha_\alpha.
\]
The main idea is now to pass to the limit when $\alpha$ goes to zero in the equality above, to conclude that $\U =  G^* \Phi$, where $\Phi$ is solution to \eqref{eq:P:Standard_weak}-\eqref{eq:P:Standard_CI}. Note that the weak convergence
\[
    G_\alpha^* \PhiAlpha_\alpha \underset{L^2(0,T;\mH)}{\rightharpoonup}   G^* \Phi
\]
would imply that $   \U  =  G^* \Phi$. It is then sufficient to prove that the weak convergence holds and that the limit $\Phi$ is solution to \eqref{eq:P:Standard_weak}-\eqref{eq:P:Standard_CI}.  \\

To prove the mentioned convergence results we first provide preliminary uniform (with respect to $\alpha$) energy estimates for the solution $\PhiAlpha_\alpha$. Following the same proof as Proposition \ref{prop:Energy}, the following estimation can be proved, 
\begin{equation} \label{eq:Equivalence:Energy}
\frac{1}{2} \| \frac{d}{dt} \PhiAlpha_\alpha(t) \|_{\mGAlpha}^2 
    + 2 \alpha \int_0^t \| \PhiAlpha_\alpha(t')  \|_{\mGAlpha}^2 \, \xdif t'
    +  \frac{1}{2} \|G_\alpha^* \PhiAlpha_\alpha(t) \|_{\mH}^2  
\lesssim  B_\alpha^2(t),     
\end{equation}
where the inequality above holds up to a constant independent of $  \PhiAlpha_\alpha $ or the source term $  u_b $, and with
\[
B_\alpha(t) = 
\sup_{s\in [0,t]}  \| e^{-\alpha s} u_b (t) \|_{H^{-1/2}(\Gamma_b)}
+  \int_0^t \| \partial_t ( e^{-\alpha s} u_b(s)) \|_{H^{-1/2}(\Gamma_b)} \, \xdif s.  
\]
For $\alpha \leq 1$, it holds 
\[
    B_\alpha(t) \leq D(t) := C(t) + \int_0^t \|  u_b(s)) \|_{H^{-1/2}(\Gamma_b)} \, \xdif s,
    \]
    where $B(t)$ is defined in \eqref{eq:def_B}. 
Therefore, the right-hand side of   \eqref{eq:Equivalence:Energy} can be replaced by a positive function independent of $\alpha$. 
For $\alpha \leq 1$, we have
\begin{equation} \label{eq:Equivalence:Energy_uniform}
 \| \frac{d}{dt} \PhiAlpha_\alpha(t) \|_{\mGAlpha}
    +   \|G_\alpha^* \PhiAlpha_\alpha(t) \|_{\mH} 
\lesssim  C(t).
\end{equation}
From \eqref{eq:P_Alpha:Weak}, we also deduce an estimate of the second order time derivative in the Hilbert space $ \mD(G_\alpha^* )' $, 
\begin{equation}\label{eq:Equivalence:strong_uniform}
    \| \frac{d^2}{dt^2} \PhiAlpha_\alpha(t) \|_{D(G_\alpha^* )'} \lesssim  C(t).
\end{equation}
Those preliminary observations allow us to state the following result.

\begin{lmm}
The functions $G_\alpha^* \PhiAlpha_\alpha$ for $\alpha \in \R^+$ converges weakly in $L^2(0,T;\mH)$ to $G^* \Phi$ when $\alpha \to 0$ with $ \Phi$ solution to \eqref{eq:P:Standard_weak}-\eqref{eq:P:Standard_CI}.
\end{lmm}

\begin{proof} 
Using the estimates \eqref{eq:Equivalence:Energy_uniform}-\eqref{eq:Equivalence:strong_uniform} and the fact that $\PhiAlpha_\alpha  $ vanishes at the initial time, we have  that $ \PhiAlpha_\alpha $, when considered as a sequence in $\alpha$, is bounded in the Hilbert space
\[
  \mathcal  W:=  L^2(0,T; \mD(G_\alpha^*) ).
\]
Note that the domain of the operator $G^*_\alpha$ is independent of $\alpha$, from the definition $\mD(G^*_\alpha) = \mD(G^*) \times \mH$. Hence, we deduce that $   \PhiAlpha_\alpha $     converges weakly -- up to a subsequence -- to a function $ \PhiAlpha $ in $\mathcal W$.  
Decomposing $\PhiAlpha = (\Phi \ {\bf V}) ^t$ and passing
to the limit, $\alpha \to 0$ in the formulation \eqref{eq:P_Alpha:Weak}, we obtain
 \begin{equation} \label{eq:Equiv:Decomp_Phi}
   \frac{d^2}{dt^2} (\Phi, \tilde \Phi )_\mG  
 + (G^* \Phi, G^* \tilde \Phi)_\mH
  = \langle u_b , \tilde \varphi \rangle_{\Gamma_b}
 \quad \forall \tilde \PhiAlpha = (\Phi \ 0)^t  \in \mD(G^*) \times \mH,
\quad \text{ in }  \mD'(0,T), 
 \end{equation}
 and 
 \begin{equation} \label{eq:Equiv:Decomp_V}
  \frac{d^2}{dt^2} ({\bf V}, \tilde {\bf V})_\mH  
  = 0
 \quad \forall \tilde \PhiAlpha = (0\ \tilde {\bf V})^t  \in \mD(G^*) \times \mH,
\quad \text{ in }  \mD'(0,T).
 \end{equation}
Hence $\Phi$ is solution to \eqref{eq:P:Standard_weak}. This almost concludes the proof. Indeed, only initial conditions must be investigated to conclude. From the energy estimate \eqref{eq:Equivalence:Energy_uniform} we also have the weak convergence in $H^1(0,T;\mGAlpha)$, therefore, for all $\tilde \PhiAlpha \in \mGAlpha,$ 
 \[
\big(\PhiAlpha_\alpha(0), \tilde \PhiAlpha   \big)_{\mGAlpha} 
= \frac{1}{T}\int_{0}^T  \frac{d}{dt}  \big( 
(t-T) ( \PhiAlpha_\alpha,  \tilde \PhiAlpha)_{\mGAlpha} 
\big) \xdif t 
\underset{\alpha \rightarrow 0}{\longrightarrow} 
\frac{1}{T}\int_{0}^T  \frac{d}{dt}  \big( 
(t-T) ( \PhiAlpha,  \tilde \PhiAlpha)_{\mGAlpha} 
\big) \xdif t 
= \big(\PhiAlpha(0), \tilde \PhiAlpha   \big)_{\mGAlpha},
 \]
 hence $ \PhiAlpha(0) = \PhiAlpha_\alpha(0) = 0.$ Similarly, thanks to the weak convergence in $H^2(0,T;  (\mD(G^*) \times \mH)' ) $ we conclude that 
 \[
\frac{d}{dt} \PhiAlpha(0) = \frac{d}{dt} \PhiAlpha_\alpha(0) = 0.
 \]
 
\end{proof}

We have shown that $\PhiAlpha_\alpha$ converges weakly to $\Phi$, that the limit $\Phi$ is solution to  \eqref{eq:P:Standard_weak}-\eqref{eq:P:Standard_CI}, and that $ G_\alpha^* \PhiAlpha_\alpha$ converges weakly to $G^* \Phi$. Then $\U = G^* \Phi$ is solution to \eqref{eq:V:Weak_1}-\eqref{eq:V:Weak_CI_1}, which proves Theorem \ref{th:Equivalence}.




\section{Numerical illustration} \label{sec:Approximation}

We now use the two formulations \eqref{eq:V:Weak_1}-\eqref{eq:V:Weak_CI_1} and \eqref{eq:P:Non_standard_weak}  to give several numerical illustrations of tsunamis and hydro-acoustic waves generated by an earthquake or landslide source. 
Three sets of simulations are presented in this section.
The first one reproduces a classical scenario of an submarine earthquake generating a tsunami and acoustic waves. 
This scenario is used for several validations: convergence analysis, validating the model by comparison with the literature,
and illustrating the equivalence of the potential-based and velocity-field formulations. 
In the second set of simulations, the equation $\U = G^* \Phi$ is used to investigate the validity of the classical irrotational assumption for the velocity field. 
The third simulation is a preliminary work towards the study of acoustic waves generated by submarine landslides. 
We compute the spectrogram of pressure recorded at a given sensor for a source located on the seabed which emits a range of frequencies typical of those observed in the field. 
This novel simulation highlights the interference pattern caused by the reflection of the acoustic waves at the surface.  

We first describe the choice of parameters. Then the discretization of the velocity-field and the potential-based simulations are presented.
Finally, for each set of simulations, we describe the scenario and provide some illustrations.

As stated in the introduction, the background functions $\rho_0, c_0$ must satisfy the following positivity property, 
\begin{equation}
    - \frac{g}{\rho_0} \frac{d\rho_0}{dz} - \frac{g^2}{c_0^2} = N^2 \geq  0. 
\end{equation}
Unless stated otherwise, we use the simplified case where $N$ and $c_0$ are chosen constant. 
This gives an ordinary differential equation satisfied by $\rho_0$, 
\begin{equation}
    \frac{d\rho_0}{dz} 
    + \rho_0
    \left( \frac{N^2}{g} + \frac{g}{c_0^2} \right) 
    = 0,
\end{equation}
hence the density profile has the form
\begin{equation}
    \rho_0(z) = \rho_0(0) \exp(- n^2  z), 
    \quad n^2 = \frac{N^2}{g} + \frac{g}{c_0^2} > 0. 
\end{equation}
The numerical values used in the simulations are $\rho_0=1000$~kg m$^{-3}$, 
$c_0=1500$~ms$^{-1}$ and $N=0.001$~s$^{-1}$. 
Note that the model can also handle more complex stratification, where the background functions are computed from a given temperature profile. This yields depth-dependent $N$ and $c_0$.
For conciseness, the case with a given temperature profile is presented in Appendix \ref{sec:A:Buoyancy}.


\subsection{Discretization}
In this part, we introduce the variational formulations to be discretized and the notations for the finite element approximation. 
The space discretization, obtained with a high-order spectral element method \cite{komatitsch_spectral_1998, cohen_higher-order_2001} is then written.
We end this part by presenting the fully-discrete scheme, which uses a second-order discretization in time.

\subsubsection{velocity-field formulation}
For the continuous problem \eqref{eq:V:Weak_1}-\eqref{eq:V:Weak_CI_1}, the velocity is sought in a subspace of $\HDiv$.
The discretization of $\HDiv$ can be done using e.g. Raviart-Thomas elements \cite{brezzi_mixed_1991}, but the problem is discretized in $V_h \subset H^1(\Omega)$ for the simplicity of implementation. 
Such a strategy is certainly not adequate for harmonic problems \cite{bonnet-ben_dhia_time-harmonic_2007}, or for transient problems with a mean flow \cite{bonnet-ben_dhia_regularisation_2006}. However, for the current problem we have not observed poor behaviour of the solution with this choice. 
To impose the boundary condition we use a Lagrange multiplier, hence, a finite dimensional space $M_h \subset L^2(\Omega)$ is introduced. The following discrete variational formulation is considered: find 
\begin{equation}
    \U_h \in C^2([0,T];V_h^d)  \quad \mbox{ and } \quad \lambda_h  \in C^0([0,T];M_h) 
\end{equation} 
solution to 
\begin{align}
\frac{d^2}{dt^2} (\U_h(t), \V_h)_\mH
+ (\Ge \U_h(t), G \V_h)_\mG
+ ( \lambda_h, \V_h \cdot \normal_b)_{L^2(\Gamma_b)}  = 0, 
& \quad \forall \,  \V_h \in V_h^d,
\quad \forall t \in [0,T], \label{eq:Numeric:Velocity}
\\
(\U_h(t) \cdot \normal_b - u_b, \mu_h)_{L^2(\Gamma_b)} = 0
&  \quad   \forall \,  \mu_h \in M_h,
\quad \forall t \in [0,T].
\label{eq:Num:Velocity_BC} 
\end{align}
The spectral element method is used for the construction of $V_h$, and $M_h$ is constructed using traces of functions in $V_h$. 
Adequate quadrature formulae (that use Gauss-Lobatto quadrature rule) are used to compute integrals, leading in particular to mass lumping: the mass matrix is diagonal. 
The semi-discrete algebraic 
problem reads
\begin{align} \label{eq:Numeric:Matrix_U}
    \frac{d^2}{dt^2} \mathbb M_\U \, U_h
   + \mathbb K_\U \, U_h
   +C_h \Lambda_h= 0, \\
   C_h^T U_h - U_{b,h} = 0,
\end{align}
where $ U_h $  and $ \Lambda_h $ are vectors of degrees of freedoms. 
The time interval $[0,T]$ is partitioned into $M$ equal intervals of length 
$\Delta t = T/M$. 
The finite-element approximation \eqref{eq:Numeric:Matrix_U} is discretized in time with a leapfrog scheme. 
We consider the sequence $\{ U_h^n  \in V_{h} \}_{n \in \{1,..., M \}}$ solution to
\begin{align} \label{eq:Num:U_Leapfrog}
    \mathbb M_\U \, 
    \frac{U_h^{n+1} - 2 U_h^n + U_h^{n-1}}{\Delta t^2}
   + \mathbb K_\U \, U_h^n +  C_h \Lambda_h^n
   = 0, \\
   C_h^T U_h^n - U_{b,h}(n\Delta t) = 0.
\end{align}
Note that in \eqref{eq:Num:U_Leapfrog} the mass matrix $ \mathbb M_\U $ is diagonal, hence easily invertible. Moreover, at each time step, the matrix $C_h^T  C_h $ must also be inverted. It is also diagonal when adequate quadrature formulae are used and thanks to our choice of having $M_h$ as a space of traces of functions of $V_h.$

\subsubsection{Potential-based formulation}
As presented in Section  \ref{sec:Potential}, we use the variational formulation \eqref{eq:P:Non_standard_weak}. 
The problem is written as a coupled system for the variables $(\varphi(t), \psi(t)) \in H^1(\Omega) \times L^2(\Omega)$ for all time $t \in [0,T]$.
By taking test function $\tilde \psi = 0$ in \eqref{eq:P:Non_standard_weak}, the variational formulation reads
\begin{multline} \label{eq:Numeric:phi}
\frac{d^2}{dt^2}
\left( \intdom \frac{\rho_0}{c_0^2} \varphi \tilde \varphi \xdif x 
+ \ints \frac{\rho_0}{g} \varphi \tilde \varphi
\right)
+ \intdom \rho_0 \left(\nabla \varphi
- \frac{N^2}{g} \varphi \,  \e{z} \right) 
\cdot 
\left(\nabla \tilde \varphi
- \frac{N^2}{g} \tilde \varphi \,  \e{z} \right) \xdif x
\\
- \intdom \rho_0 N \psi  \e{z}
\cdot 
\left(\nabla \tilde \varphi
- \frac{N^2}{g} \tilde \varphi \,  \e{z} \right) \xdif x
=
- \intb \frac{\rho_0}{g} u_b \tilde \varphi, 
\end{multline}
and by taking test function $\tilde \varphi = 0$ in \eqref{eq:P:Non_standard_weak}, the variational formulation reads
\begin{equation} \label{eq:Numeric:psi}
\frac{d^2}{dt^2}
\intdom \rho_0 \psi \tilde \psi \xdif x 
- \intdom \rho_0 \left( 
\nabla \varphi
- \frac{N^2}{g} \varphi \,  \e{z} 
\right) 
\cdot 
\left(N \tilde \psi \,  \e{z} \right) \xdif x
+ \intdom \rho_0 N^2 \psi \tilde \psi \xdif x
= 0.
\end{equation}


We introduce a finite-dimensional space $L_h  \subset L^2(\Omega)$.
The following bilinear forms are defined for $ \varphi_h, \tilde \varphi_h \in V_h$ and 
$\psi_h, \tilde \psi_h \in L_h$, 
\begin{align}
m_{\varphi}(\varphi_h, \tilde \varphi_h) 
    &= (\frac{\rho_0}{c_0^2} \varphi_h, \tilde \varphi_h )_{L^2(\Omega)}
    + (\frac{\rho_0}{g} \varphi_h, \tilde \varphi_h )_{L^2(\Gamma_s)},
\quad 
m_{\psi}(\psi_h, \psi_h) 
    = (\rho_0 \psi_h, \tilde \psi_h)_{L^2(\Omega)}, 
\\
a_{\varphi}(\varphi_h, \varphi_h)
    &= (\nabla \varphi_h - N^2 \varphi_h \, \e{z}, 
    \nabla \varphi_h - N^2 \varphi_h \, \e{z})_{\mH}, 
\quad  
a_{\psi}(\psi_h, \tilde \psi_h) 
    = (\rho_0 N^2 \psi_h, \tilde \psi_h)_{L^2(\Omega)},
\\[5pt]
c(\psi_h, \tilde \varphi_h)
    &= (- \rho_0 N \psi_h \e{z}, 
    \nabla \tilde \varphi_h - N \tilde \varphi_h \e{z})_{\mH}.
\end{align}
The semi-discrete variational formulation reads then:
find 
\begin{equation} 
(\varphi_h, \psi_h) \in 
C^2 \left( [0,T]; V_h \times L_h  \right)
\end{equation} solution  to
\begin{align}
& \frac{d^2}{dt^2} m_{\varphi}(\varphi_h, \tilde \varphi_h)
+ a_{\varphi}(\varphi_h, \tilde \varphi_h) 
+ c(\psi_h, \tilde \varphi_h)
=   - (\rho_0 u_{b} /g , \tilde \varphi_h)_{L^2(\Gamma_b)}, 
\quad \forall \tilde \varphi_h \in V_h,
\quad \forall t \in [0,T],
\label{eq:Numeric:FE_phi}
\\[4pt]
& \frac{d^2}{dt^2} m_{\psi}(\psi_h, \tilde \psi_h)
+ a_{\psi}(\psi_h, \tilde \psi_h) 
+ c (\tilde \psi_h, \varphi_h) = 0, 
\quad \quad  \quad  \quad  \quad \quad  \quad  \quad  \; \;  \forall \tilde \psi_h \in L_h, 
\quad \forall t \in [0,T], 
\label{eq:Numeric:FE_psi}
\end{align}
An approximated velocity $\U_h$ is computed from the variables $(\varphi_h, \psi_h)$ using formally the relation ${\U = G^* \Phi}$. 
At the discrete level we solve
\begin{equation} \label{eq:Numeric:Phi_to_U}
    (\U_h, \V_h)_{\mH}
    =
    b_{\varphi}(\varphi_h, \V_h)
    +  b_{\psi}(\psi_h, \V_h), 
    \quad \forall \, V_h^d \in \mV_{h},
\end{equation}
where $b_{\varphi}:V_h \times V_{h}^d \to \R$ and $b_{\psi, h}:L_h \times V_{h}^d \to \R$ are the following bilinear forms,
\begin{align} 
    b_{\varphi}(\varphi_h, \V_h) =
    - (\nabla \varphi_h - \frac{N^2}{g}  \e{z} \varphi_h,  \V_h)_{\mH}, 
    \quad 
     b_{\psi, h}(\psi_h, \V_h) =
    (N \psi_h \e{z} , \V_h)_{\mH}.
\end{align}
In practice the spectral element method is used to construct $V_h$ and $L_h$ and all integrals are approximated using an adequate quadrature formula.
For the potential-based formulation, the decomposition on each basis yields
the global vectors $\Phi_h$, $\Psi_h$ and $u_{b,h}$, as well as the mass matrices 
$\mathbb M_\varphi, \mathbb M_\psi$ and $\mathbb M_b$, the stiffness matrices 
$\mathbb K_\varphi, \mathbb K_\psi$, 
and the interaction matrix 
$\mathbb C_{\psi,\varphi}$. 
We also obtain the matrices $\mathbb B_\varphi, \mathbb B_\psi$ used to deduce the velocity from the potential, see Equation \eqref{eq:Numeric:Phi_to_U}.
The semi-discrete system corresponding to the finite-element approximation \eqref{eq:Numeric:FE_phi}-\eqref{eq:Numeric:Phi_to_U} reads
\begin{align} 
& \frac{d^2}{dt^2} \mathbb M_\varphi \Phi_h
+ \mathbb K_\varphi \, \Phi_h
+ \mathbb C_{\psi \varphi} \, \Psi_h
= - \mathbb M_b \, U_{b,h} , \label{eq:Numeric:Matrix_phi}
\\
& \frac{d^2}{dt^2} \mathbb M_\psi \Psi_h
+ \mathbb K_\psi \, \Psi_h
+ \mathbb C_{\psi \varphi}^T \, \Phi_h 
= 0, 
\label{eq:Numeric:Matrix_psi}
\\[4pt]
& \mathbb M_\U \, U_h
=
\mathbb B_\varphi \Phi_h
+ \mathbb B_\psi \Psi_h.
\end{align}
By using the Gauss-Lobatto quadrature rule for the integral approximations, we obtain that the mass matrices $ \mathbb M_\varphi, \mathbb M_\psi$ and $\mathbb M_b$ are diagonal. 
Again, the time interval $[0,T]$ is partitioned into $M$ equal intervals of length 
$\Delta t = T/M$ and for the potential-based formulation, we consider the sequences $\{ \Phi_h^n  \in V_{h} \}_{n \in \{1,... M \}}$ and 
$\{ \Psi_h^n  \in L_{h} \}_{n \in \{1,... M \}}$ solution to
\begin{align}
\mathbb M_\varphi 
& \frac{\Phi_h^{n+1} - 2 \Phi_h^n + \Phi_h^{n-1}}{\Delta t^2}
+ \mathbb K_\varphi \, \Phi_h^n
+ \mathbb C_{\varphi \psi} \, \Psi_h^n 
=  - \mathbb M_b \, U_{b,h}(n \Delta t), \label{eq:Num:P_Leapfrog_1}
\\[4pt]
& \mathbb M_\psi 
\frac{\Psi_h^{n+1} - 2 \Psi_h^n + \Psi_h^{n-1}}{\Delta t^2}
+ \mathbb K_\psi \, \Psi_h^n
+ \mathbb C_{\varphi \psi}^t \, \Phi_h^n
=  0,
\end{align}
and the velocity is then computed with 
\begin{equation} \label{eq:Num:P_Leapfrog_3}
    \mathbb M_\U \, U_h^n
    =
    \mathbb B_\varphi \Phi_h^n
    + \mathbb B_\psi \psi_h^n.
\end{equation}
In the following, the displacement will be sometimes used. The displacement is computed by integrating in time the velocity.

\subsection{First simulation: waves generation from a submarine earthquake}
In order to validate and compare the formulations \eqref{eq:Num:U_Leapfrog} and 
\eqref{eq:Num:P_Leapfrog_1}-\eqref{eq:Num:P_Leapfrog_3}, we reproduce the test case presented in \cite{sammarco_depth-integrated_2013}.
A tsunami and hydro-acoustic waves are generated by a seabed elevation due to a submarine earthquake, in a 2D domain on a flat topography. 
In \cite{sammarco_depth-integrated_2013}, the flow is assumed irrotational with a velocity potential $\phi$. 
The distance between the seabed and the mean water level is denoted $h(x,y,t)$.
The model, solved with a finite elements method, reads 
\[
\frac{\partial^2 \phi}{\partial t^2} - c^2 \Delta \phi = 0, 
\]
where $c=1500$~ms$^{-1}$ is constant, and with boundary conditions (with $g=9.81$~ms$^{-2}$)
\[
\frac{\partial^2 \phi}{\partial t^2} 
+ g \frac{\partial \phi}{\partial z} 
= 0, \text{ at } z=h(x,y,t), \quad
\nabla \phi \cdot \normal
= \frac{\partial h}{\partial t}, \text{ at } z=0.
\]

\begin{figure}
    \centering
    \includegraphics[width=\textwidth]{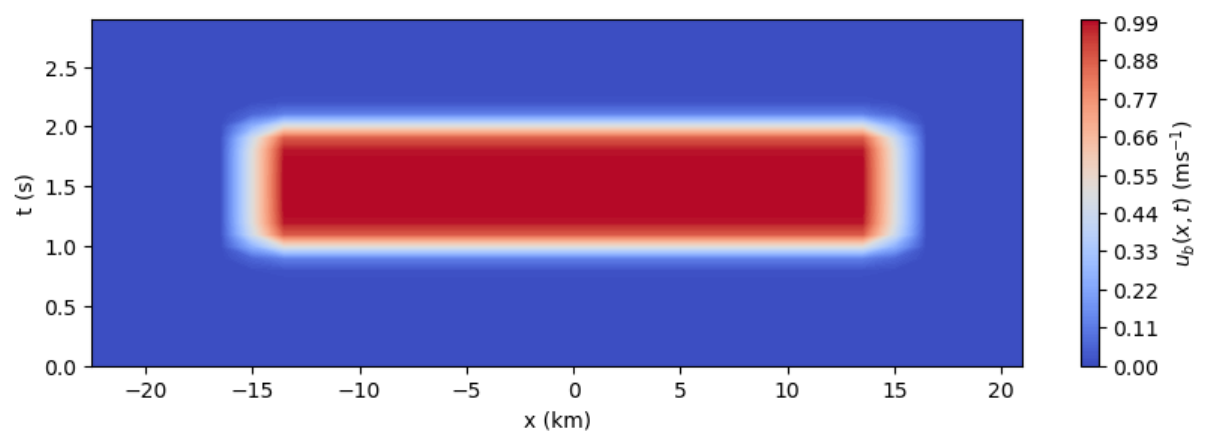}
    \caption{The function $u_b(x,t)$ describing the earthquake used in the first simulation to generate acoustic waves and tsunamis. 
    }
    \label{fig:Simu1:Source}
\end{figure}

In the simulations, we use the following source: 
$u_b(x,t) = f(x) g(t)$, where $f$ and $g$ are smoothed rectangular functions,
\begin{equation}
f(x) = 
\frac{1}{1+e^{- s_x (x - r_x/2)}}
- \frac{1}{1+e^{- s_x (x + r_x/2)}},
\quad
g(t) = \frac{1}{1+e^{- s_t (t-t_0)}}
- \frac{1}{1+e^{- s_t (t - t_0 - r_t)}}.
\end{equation}
see Figure \ref{fig:Simu1:Source}.
Since no conditions are imposed on the lateral boundaries, the computational domain is wider to avoid reflections inside the domain of interest.
The numerical values are given in Table \ref{tab:Simu1:Values}.

\begin{table}[]
    \centering
    \begin{tabular}{c | l }
        Parameter  & Value       \\
        \hline
        Domain height & 1.5 km  \\
        $s_x$ & 150 m$^{-1}$    \\
        $r_x$ & 30 km           \\
        $t_0$ & 2 s             \\  
        $s_t$ & 4 s$^{-2}$      \\ 
        $r_t$ & 1 s             \\
    \end{tabular}
    \caption{Parameter values for the first simulation: hydro-acoustic waves and tsunami generation by an earthquake.  
    }
    \label{tab:Simu1:Values}
\end{table}


\subsubsection{Convergence analysis}
For each formulation, several elements orders are tested. 
The aim is to show that the simulations converge and to select parameters that are precise enough and not too costly.
The tsunami is simulated for 50 s and the computational domain is 101 km long.
To evaluate whether convergence is reached, we plot the time evolution of the free surface displacement at a point $x = 50$~km away from the source center.
The convergence is estimated visually and is assumed reached when increasing the order does not change the curves.
The finite element orders are denoted \verb|Px|, \verb|Pz| and the number of mesh subdivision in the $x$ and in the $z$ coordinates are respectively denoted \verb|Nx|, \verb|Nz|. 
Several combinations of \verb|Px| and \verb|Pz| are tested while keeping \verb|Nx| and \verb|Nz| fixed. 
The total number of Degrees of Freedom (NDoFs) for each case is indicated in each surface displacement plot. 

For the velocity-field scheme, the surface displacement is shown in Figure \ref{fig:Convergence_Primal}. 
The convergence is reached for \verb|Px|~=~4, \verb|Pz|~=~5.
For the potential-based scheme, the surface displacement is shown in Figure~\ref{fig:Convergence_Dual1} and we see that the orders with \verb|Px|~$\geq$~3 converge.
However, when taking  \verb|Px|~$=$~3, \verb|Pz|~$=$~3, we have observed increasing oscillations for longer times (not shown here).
These oscillations come from the seabed displacement at $x=0$ km, as shown in Figure~\ref{fig:Convergence_Dual2}. 
The convergence is actually reached for \verb|Px|~=~3, \verb|Pz|~=~5.
Such oscillations are not observed for the velocity-field formulation.

\begin{figure}
\centering
	\includegraphics[width=\textwidth]{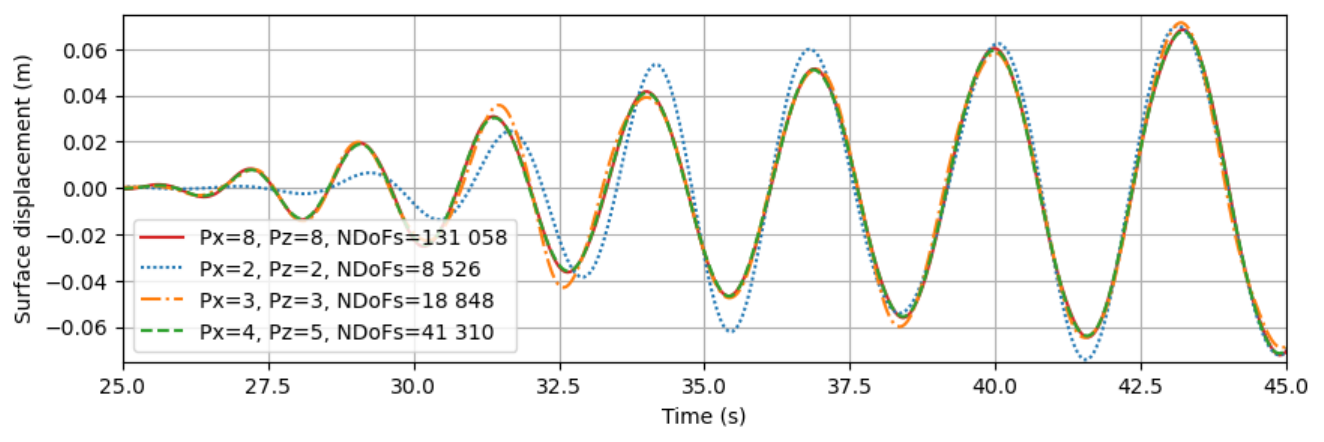}
	\caption{Surface vertical displacement at $x=50$ km for the velocity formulation at different elements orders.}
\label{fig:Convergence_Primal}
\end{figure} 

\begin{figure}
\centering
	\begin{subfigure}{\textwidth}
	\includegraphics[width=\textwidth]{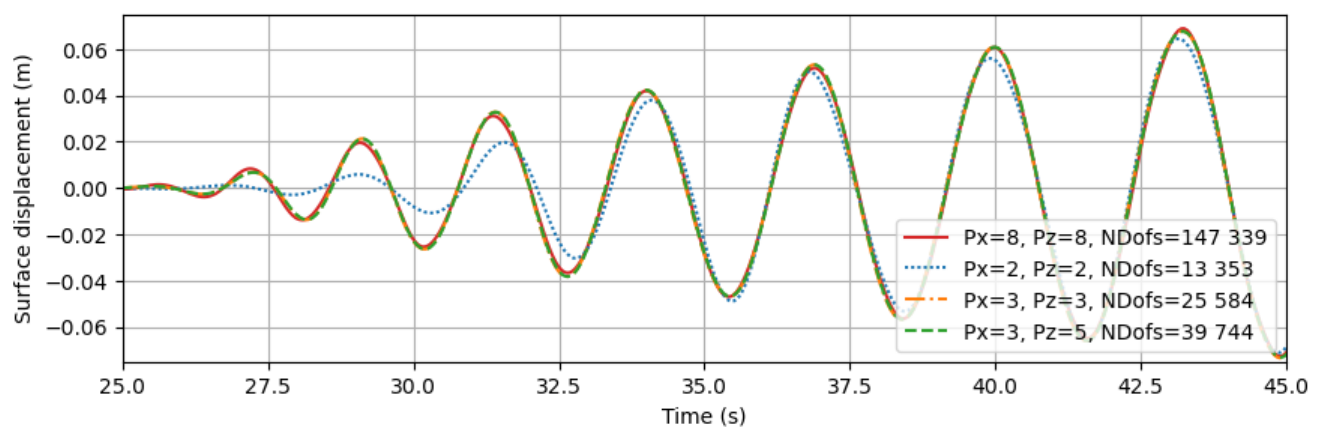}
	\caption{Surface vertical displacement at $x=50$ km.}
\label{fig:Convergence_Dual1}
\end{subfigure} 
\begin{subfigure}{\textwidth}
	\includegraphics[width=\textwidth]{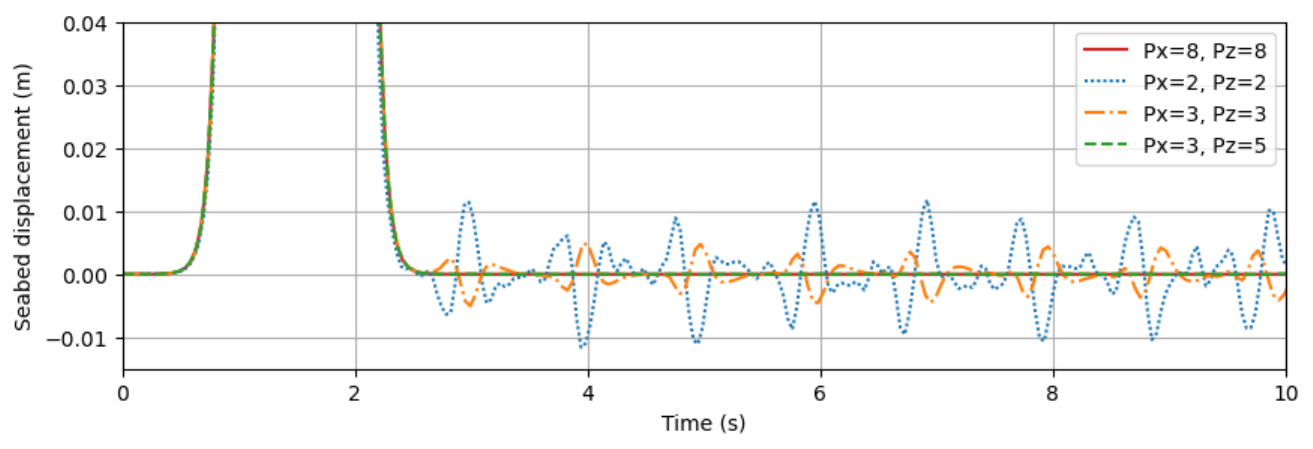}
	\caption{Seabed vertical displacement at $x$=0 km.}
\label{fig:Convergence_Dual2}
\end{subfigure}
\caption{Displacement of (a) the water free surface at $x=50$ km, (b) the seabed at $x=0$ km, obtained from the potential-based formulation for several element orders.}
\end{figure}

\subsubsection{Comparison and snapshots}
We now run the simulation for a longer time and compare it with the results from \cite{sammarco_depth-integrated_2013}. 
The simulation obtained with the velocity-field formulation and the potential-based formulations are first compared with \cite{sammarco_depth-integrated_2013}, then compared to each other. 
Finally we present some snapshots of the simulation. 

\begin{table}[]
    \centering
    \begin{tabular}{|c|c|c|c|c|c|}
        \hline
         Nx   & Nz & Px & Pz & NDoFs    & Computational time  \\ \hline
         1051 & 10 & 4  & 5  & 428 910 & 771 s \\ \hline
    \end{tabular}
    \caption{Parameters values for the velocity-field simulation.}
    \label{tab:Simu1:ParamVelocity}
\end{table}
\begin{table}[]
    \centering
    \begin{tabular}{|c|c|c|c|c|c|}
        \hline
         Nx   & Nz & Px & Pz & DoFs  & Computational time  \\ \hline
         1051 & 10 & 3 & 5 & 413 094 & 1197 s \\ \hline 
    \end{tabular}
    \caption{Parameters values for the potential-based simulation.}
    \label{tab:Simu1:ParamPotential}
\end{table}

The tsunami is simulated for 1000 s, and the computational domain is 1576.5 km long. 
The element orders, the DoFs and the  computational time are given in Table \ref{tab:Simu1:ParamVelocity} and \ref{tab:Simu1:ParamPotential} respectively for the velocity-field and potential-based formulations.
The potential-based formulation requires more computational time than the velocity-field formulation.
This additional time is due to the displacement calculation: the displacement is obtained by integrating the velocity in time, and to obtain a precise integration, the velocity is computed at each time step.  
Hence the system \eqref{eq:Num:P_Leapfrog_3} must be solved at each time step, leading to an additional computational cost. 
The code could be accelerated by using an efficient time integration rule which would not require to compute the velocity at each time step, or by directly solving the potential-based problem associated to the displacement.  \\

For comparison with the literature, we consider the surface vertical displacement at a point $x=50$~km evolving with time.
The vertical displacement computed with the velocity-field model is superimposed to the result from \cite{sammarco_depth-integrated_2013} in Figure \ref{fig:Comparison:Sammarco}. 
Both curves have a high-frequency component starting at time $t \sim 20$~s, namely the acoustic waves, and a low-frequency wave starting at $t \sim 300$~s
, which is the tsunami. 
The tsunami arrival times are in good agreement between the two curves. 
It should be noted that since we do not have the exact description of the boundary condition used in \cite{sammarco_depth-integrated_2013}, the earthquake description slightly differs between the two simulations. This could explain the phase difference between both displacements. 
\\

To compare the potential-based and the velocity-field formulations, we plot in Figure~\ref{fig:Comparison} the same vertical displacement obtained from both formulations.
The orange curve is the displacement $d_\Phi$ computed from the potential formulation, and the blue curve is the displacement $d_\U$ computed from the velocity formulation. 
The curves are almost superimposed, which illustrates the fact that both formulations are equivalent.
Figure \ref{fig:Error} shows the error $|d_\Phi - d_\U|$ between the potential-based and the velocity-field formulations. The error has a magnitude of less than 5 percents, and it oscillates with a period of approximately 4 seconds. The error averaged over a period is plotted in the same figure. We see that the averaged error slowly increases, which could be caused by the time integration: small errors add up and become non-negligible for large times. \\

Finally, to give some new illustration of the wave propagation, we show in Figure \ref{fig:Snapshots_Dual} snapshots of the Eulerian domain near the earthquake at several times. The Eulerian domain is obtained by deforming the domain at rest with the displacement vector field. The colors correspond to the values of the vertical displacement. On each snapshot, we see the permanent deformation of the seabed over $[-10 ;10]$~km caused by the earthquake. 
The surface deformation is due to the hydro-acoustic waves propagating for early times ($t=49$~s), then in the next two snapshots ($t=147$~s and $t=343$~s), we see the tsunami propagating in both directions. In the last snapshot ($t=833$~s), the tsunami 
is away from the considered domain and it remains only acoustic waves.
\begin{figure}
\centering
	\includegraphics[width=\textwidth]{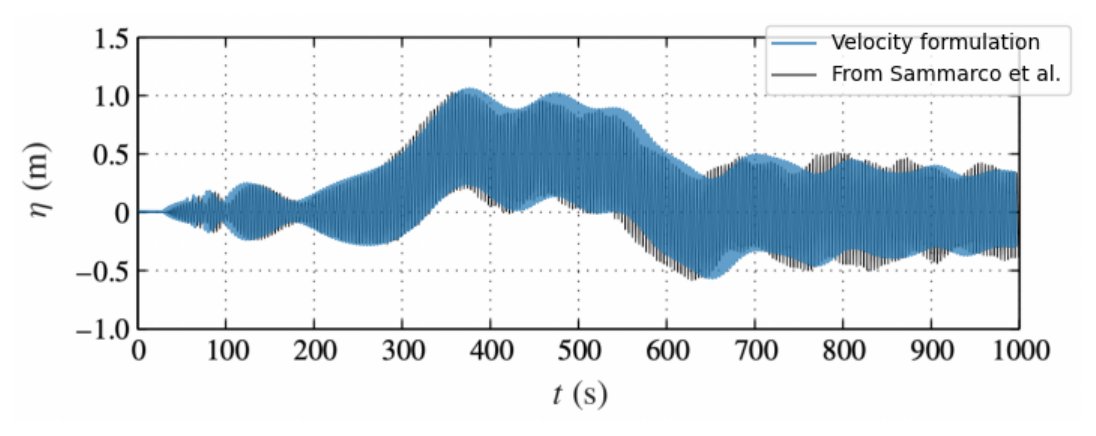}
	\caption{Vertical displacement at $x=$50 km, computed from the velocity-field (blue) superimposed to the result in \cite{sammarco_depth-integrated_2013}.}
\label{fig:Comparison:Sammarco}
\end{figure}
\begin{figure}
\centering
\begin{subfigure}{0.9\textwidth}
    \includegraphics[width=\textwidth]{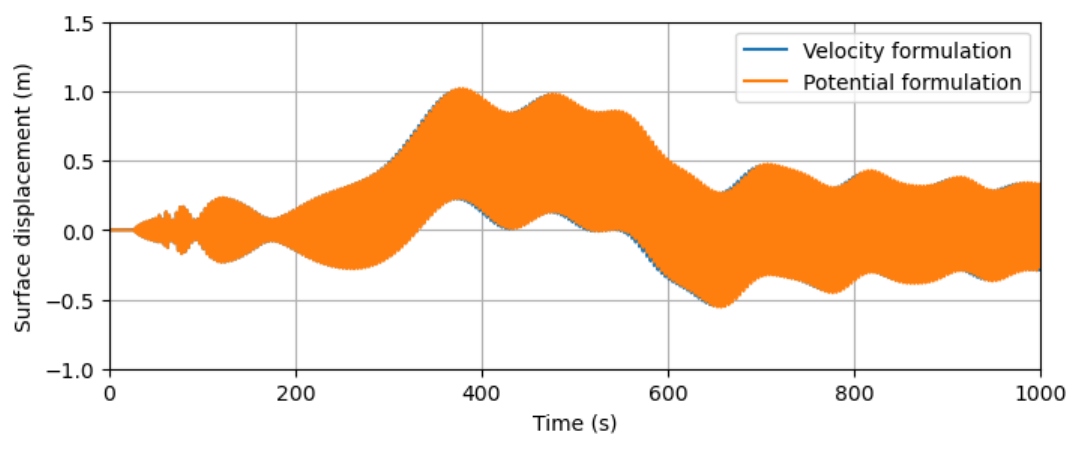}
    \caption{Vertical displacement at $x=$50 km obtained with the velocity-field (blue) and the potential-based (orange) formulations.}
    \label{fig:Comparison}
\end{subfigure}
\begin{subfigure}{0.9\textwidth}
    \includegraphics[width=\textwidth]{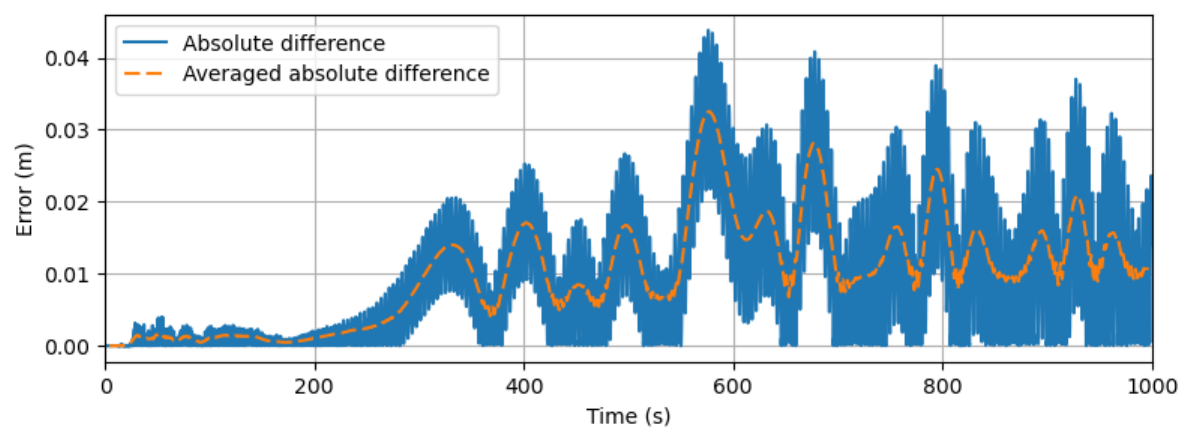}
    \caption{Absolute difference (blue) and averaged absolute difference (orange) from Figure (a).}
    \label{fig:Error}
\end{subfigure}
\caption{Comparison of the velocity-field and the potential-based formulations: (a) superposition of the displacements, (b) absolute error.}
\end{figure} 

\begin{figure}
\centering
	\includegraphics[width=\textwidth]{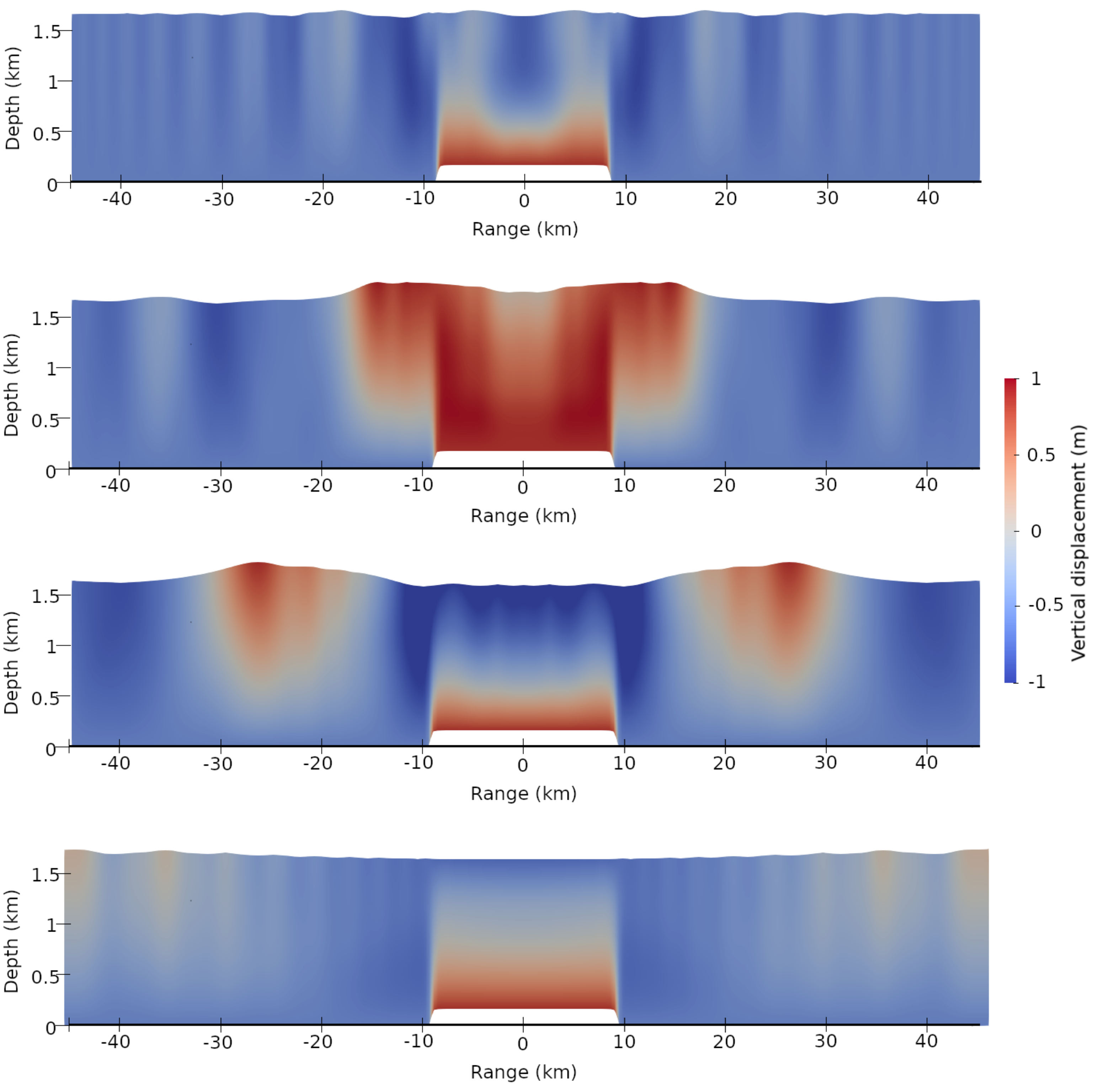}
	\caption{
	Snapshots of the water displacement obtained with the potential formulation in the first simulation. The acoustic waves and the tsunami generated by a submarine earthquake induce a water displacement.
	The snapshots are taken at the following times, from top to bottom: $t=$~49.0~s, $t=$~147.1~s,  $t=$~343.3~s, $t=$~833.8~s.
	}
\label{fig:Snapshots_Dual}
\end{figure}

\subsection{Second simulation: rotational component of the fluid velocity}
In this part, we study the validity of the traditional assumption that the velocity flow is irrotational. This assumption is made for example in \cite{longuet-higgins_theory_1950, sammarco_depth-integrated_2013, eyov_progressive_2013}. 
The generalized potential yields a decomposition of the fluid velocity into an main irrotational component $\nabla \varphi$ and a remainder $\U_r$, 
\begin{equation} \label{eq:Numerics:Decomposition}
\U = - \nabla \varphi + \U_r, 
\quad \U_r = U_r \, \e{z} =  N \left( \psi + \frac{N}{g} \varphi \right) \e{z}.    
\end{equation}
We recall that $N = \sqrt{- g \rho_0'(z)/\rho_0(z) - g^2/c_0^2(z)}$ is the buoyancy frequency \cite{dubois_acoustic_2023, gill_atmosphere-ocean_1982}.
In the case $N \equiv 0$, the decomposition \eqref{eq:Numerics:Decomposition} shows that the flow is irrotational and the potential-based formulation \eqref{eq:P:PDE_Vol}-\eqref{eq:P:PDE_Surf} simplifies to the classical model \eqref{eq:Intro:LH1}-\eqref{eq:Intro:LH2}.
Hence the potential-based formulation can be seen as a generalization of the system \eqref{eq:Intro:LH1}-\eqref{eq:Intro:LH2} to the case of a non-irrotational velocity. 
In this section, we use the potential-based formulation to quantify the contribution of the remainder $\U_r$ when the velocity is not strictly irrotational.

\begin{rmrk}
The remainder $\U_r$ defined in Equation \eqref{eq:Numerics:Decomposition} is not exactly the rotational part of the velocity as it cannot be written as $\U_r = \nabla \times \bm \psi_r$. 
\end{rmrk}

\label{sec:Simu2}
\begin{table}[]
    \centering
    \begin{tabular}{ c | l  || c  | l } 
        Parameter & Value  & Parameter & Value      \\
        \hline
        Domain width        & 15 km     & $a$       &  150 m        \\
        Domain height       & 1.5 km    & $x_0$     & 7.5 km        \\
        \verb|Px|           & 8         & $s_x$     & 4 10$^{-5}$ m$^{-2}$   \\
        \verb|Pz|           & 8         & $t_0$     & 2 s           \\
        \verb|Nx|           & 100       & $s_t$     & 4 s$^{-2}$    \\
        \verb|Nz|           & 10        & $f_x$     & 0.07 m$^{-1}$ \\           
        PML thickness       & 1.5 km    &   $k_x$   & 0.03 m$^{-1}$ \\
        $N_\text{const}$    & 0.001 s$^{-1}$ & $r_x$& 1.5 km        \\
                            &           & $b$       & 300 m         \\
    \end{tabular}
    \caption{Parameter values used in the second set of simulations for various earthquake scenarios: with and without topography, and using two different profiles for $N$. 
    }
    \label{tab:Simu2:Values}
\end{table}

We consider a scenario where one portion of the seafloor goes up and another portion goes down. 
This scenario is obtained with a source of the form $u_b(x,t) = f(x) \, g(t)$ where  
$f(x)$ is the derivative of a Gaussian and $g(t)$ is a Ricker function, 
\begin{align}
f(x) 
= - 2 a\, s_x (x-x_0) e^{- (x-x_0)^2 s_x },
\quad 
g(t) 
= (4 s_t^2 (t-t_0)^2 - 2s_t) e^{-(t-t_0)^2 s_t},
\end{align}
see Figure \ref{fig:Simu2:Source}.
\begin{figure}
    \centering
    \includegraphics[width=\textwidth]{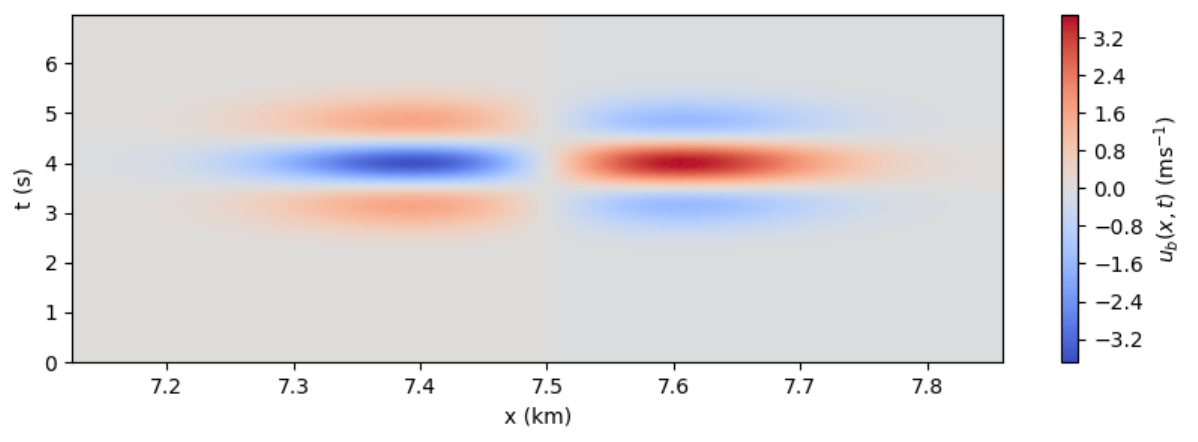}
    \caption{The function $u_b(x,t)$ describing the seabed displacement for the second simulation. The same function $u_b$ is used for every earthquake scenario in this section.}
    \label{fig:Simu2:Source}
\end{figure}
The potential-based simulation is used, and to avoid non-physical reflections, Perfectly Matched Layers (PML) were implemented.
Two cases are tested: (1) with an initially flat seabed and with a constant buoyancy frequency $N_\text{const}$, 
(2) with an initial topography $z_b(x)$ and a depth-dependant $N(z)$.
The topography $z_b$ consists of several bumps and flattens in the PML domain, 
\begin{equation} \label{eq:Simu2:Topo}
z_b(x) = b \, (1 + \sin(k_x x)) \left( \frac{1}{1+e^{- f_x (x - r_x/2)}}
- \frac{1}{1+e^{- f_x (x + r_x/2)}}\right).
\end{equation}
The profile $N(z)$ is obtained for profiles $c_0(z), \rho_0(z)$ typically found in the ocean and described in Appendix \ref{sec:A:Buoyancy}. 
The simulation parameters are described in Table \ref{tab:Simu2:Values}.
For each case, we show a snapshot at time $t=8$~s of several quantities: the irrotational component magnitude $| \nabla \varphi |$, the remainder $U_r$ and the ratio $| U_r |/ |\U |$. 
Since all quantities fluctuate in time, we also show the time averaging of  $| U_r |/ |\U |$ on each DoFs over the whole simulated time. 
\\

\begin{figure}
	\centering
	\begin{subfigure}{\textwidth}
	\includegraphics[width=\textwidth]{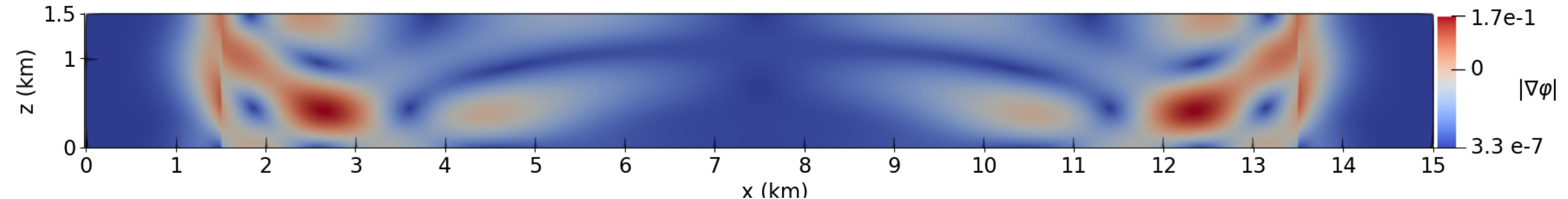}
	\caption{Main irrotational component magnitude $|\nabla \varphi|$}
	\label{fig:NConst_Irrotational}
	\end{subfigure}
	\begin{subfigure}{\textwidth}
    \includegraphics[width=\textwidth]{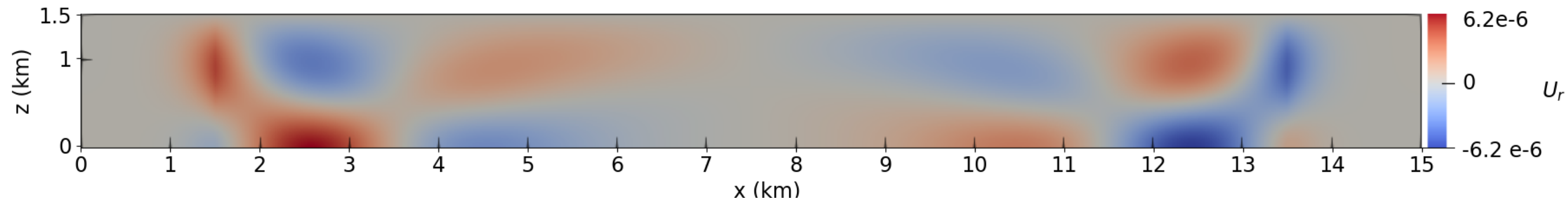}
     \caption{Remainder $U_r$}
	\label{fig:NConst_Remainder}
	\end{subfigure}	
	\begin{subfigure}{\textwidth}
	\includegraphics[width=\textwidth]{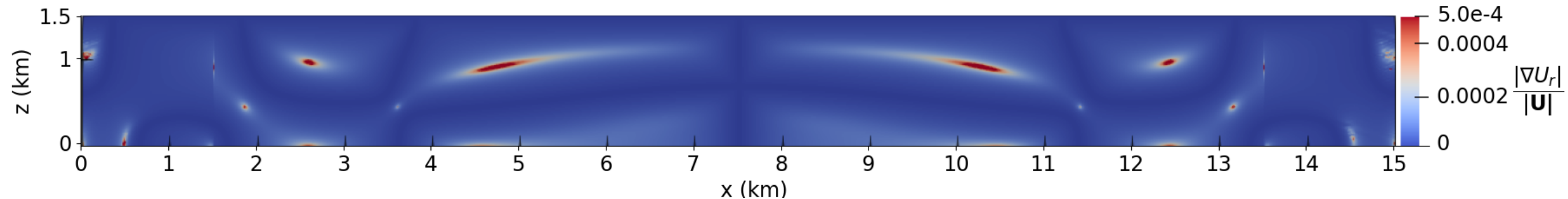}
	\caption{ relative magnitude $|U_r|/| {\bf U} |$}
	\label{fig:NConst_Ratio}
	\end{subfigure}
	\begin{subfigure}{\textwidth}
    \includegraphics[width=\textwidth]{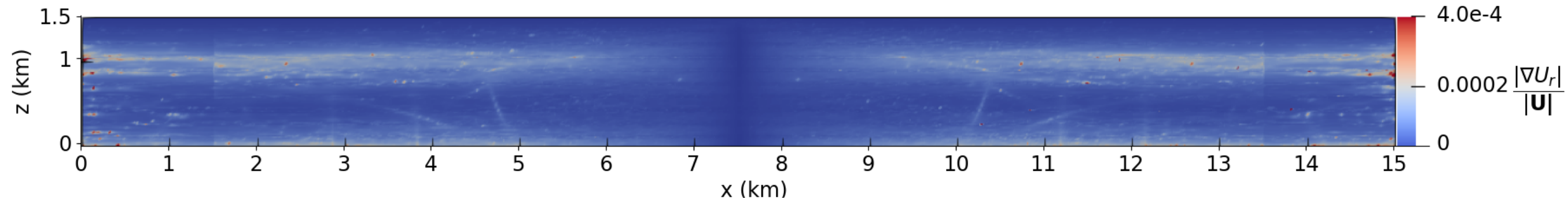}
     \caption{Time averaging of  $|U_r|/| {\bf U} |$}
     \label{fig:NConst_RatioAvg}
	\end{subfigure}
	\caption{Earthquake on a flat seabed and with a constant $N\equiv 0.001$ s$^{-1}$.
	Snapshots at $t=8s$ of: (A) the irrotational component magnitude $|\nabla \varphi|$, (B) the remainder $U_r$ and (C) the relative magnitude between the remainder and the total velocity $|U_r|/|\U|$. The time averaging over 120 s of $|U_r|/|\U|$ in shown in (D).
	}
	\label{fig:Simu2:NConst}
\end{figure}
\begin{figure}
	\centering
	\begin{subfigure}{\textwidth}
	\includegraphics[width=\textwidth]{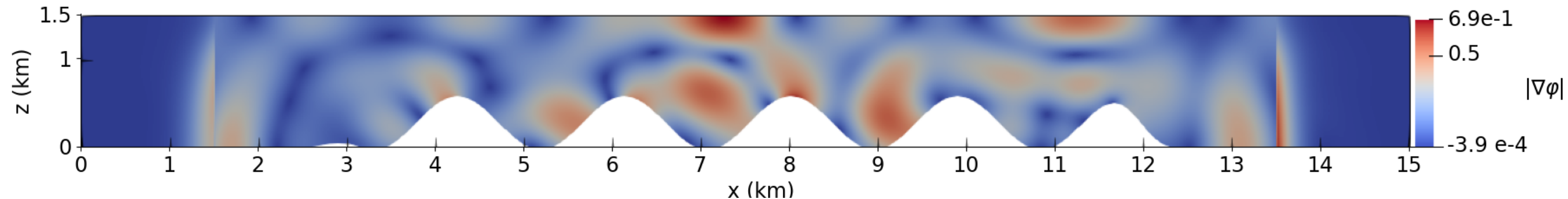}
	\caption{Main irrotational component magnitude $|\nabla \varphi|$}
	\end{subfigure}
	\begin{subfigure}{\textwidth}
    \includegraphics[width=\textwidth]{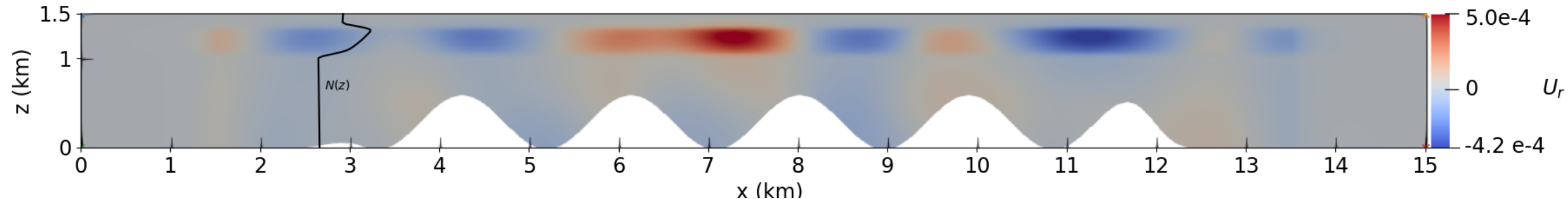}
     \caption{Remainder $U_r$}
     \label{fig:NEOS_Remainder}
	\end{subfigure}
	\begin{subfigure}{\textwidth}
    \includegraphics[width=\textwidth]{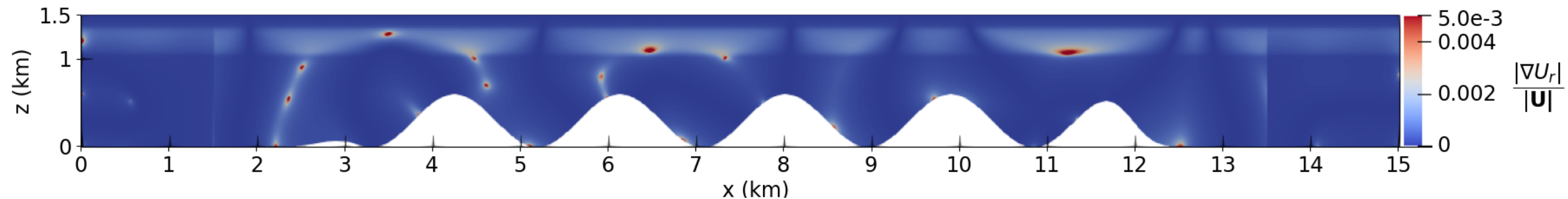}
     \caption{Ratio $|U_r|/| {\bf U} |$}
     \label{fig:NEOS_Ratio}
	\end{subfigure}
	\begin{subfigure}{\textwidth}
    \includegraphics[width=\textwidth]{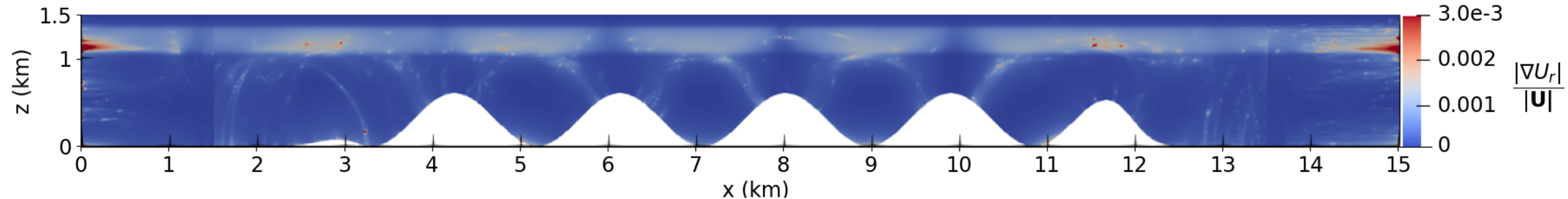}
     \caption{Time averaging of  $|U_r|/| {\bf U} |$}
     \label{fig:NEOS_Ratio_Avg}
	\end{subfigure}
	\caption{Earthquake on a topography and with a depth-dependent $N(z)$.
	Snapshots at $t=8s$ of: (A) the irrotational component magnitude $|\nabla \varphi|$, (B) the remainder $U_r$ and (C) the relative magnitude between the remainder and the total velocity $|U_r|/|\U|$. The time averaging over 120 s of $|U_r|/|\U|$ in shown in (D). The profile $N(z)$ (black curve) is shown on (B).}
	\label{fig:Simu2:NEOS}
\end{figure}

The case $N \equiv N_\text{const}$ is shown in Figure \ref{fig:Simu2:NConst}. 
From the time averaging (Figure \ref{fig:NConst_RatioAvg}), we see that the remainder is at least four orders of magnitude smaller than the main irrotational component, but that the ratio is not homogeneous in space. 
Its extremal values are on the seabed, on the upper half of the domain and away from the source location. \\

The case $N(z)$ is shown in Figure \ref{fig:Simu2:NEOS}. 
Note that in this case, the profiles $c_0(z)$ and $\rho_0(z)$ also vary with depth, see Appendix \ref{sec:A:Buoyancy}.
For reference, the profile $N(z)$ is plotted in Figure \ref{fig:NEOS_Remainder}. 
As before, the rotational component is much smaller than the irrotational component. 
The more complex profile $N(z)$ and the topography have an impact on the distribution of $U_r$. 
Indeed, the extreme values of $U_r$ are concentrated around the location of maximal $N(z)$ and on the off-peak parts of the seabed. 
We also notice that for both cases (Figures \ref{fig:Simu2:NConst} and \ref{fig:Simu2:NEOS}) the rotational component vanishes near the surface.  
Hence the irrotationality assumption seems verified for flows near the surface even when the inner flow is not strictly irrotational.
\\

For the simulations presented here, the source time function $g(t)$ is a Ricker function, so that the time average of the source is zero. 
Other simulations made with a source having a non-zero time average show tthat the remainder $\U_r$ and the ratio $|U_r|/|\U|$ have their maximal values near the source location (see Appendix \ref{sec:A:Irrotational}).
Moreover, both quantities increase with time so that for long times ($t>100$~s)  and close to the source, the ratio $|U_r|/|\U|$ is of the same magnitude order as the irrotational component $\nabla \varphi$, even for small values of $N$.
Hence the approximation $\U = \nabla \varphi$ is not uniform in space and time and can become invalid near the source for particular choices of source functions.
\\

\begin{figure}
	\centering
    \includegraphics[width=\textwidth]{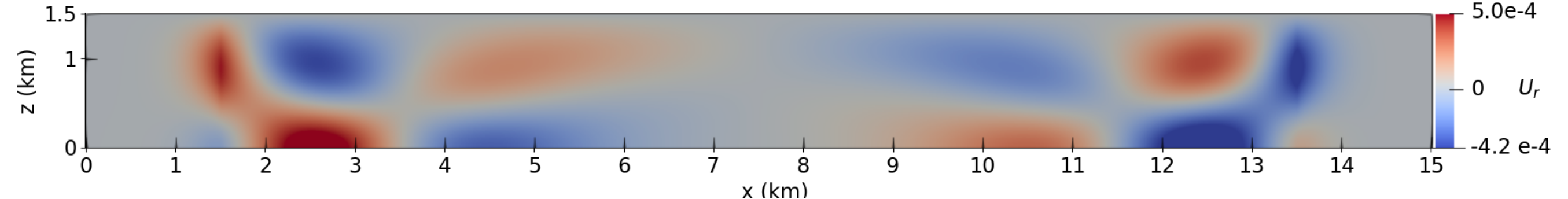}
	\caption{Earthquake on a flat seabed and with a constant $N\equiv 0.01$ s$^{-1}$:
	snapshots at $t=8s$ of the remainder $U_r$ }
	\label{fig:Simu2:NConst2}
\end{figure}
Finally, to evaluate how the values of the buoyancy frequency affect the remainder $U_r$, we run the same scenario with $N \equiv 10 \, N_\text{const}$. A snapshot at time $t=8$~s of the remainder is shown in Figure \ref{fig:Simu2:NConst2}. 
We see that $\U_r$ is distributed very similarly to the case $N \equiv N_\text{const}$, but is 100 times larger (compare to Figure \ref{fig:NConst_Remainder}). 
It seems that $\U_r$ scales approximately as $N^2$. 
From the equation \eqref{eq:Numerics:Decomposition} defining $\U_r$, this observation could indicate that the component $N^2/g \, \varphi$ is dominant over the component $N \psi$.


\subsection{Third simulation: interference patterns for submarine landslides}
In this third scenario, we use the model to produce new simulations of waves generated by an submarine landslide.
Many studies have been conducted on the generation of hydro-acoustic waves by submarine earthquakes (see \cite{gomez_near_2021, auclair_theory_2021, cecioni_tsunami_2014} and the references therein). 
However, wave generation from submarine landslides have received less attention. 
In \cite{caplan-auerbach_hydroacoustic_2014}, it is highlighted that the hydro-acoustic signals generated by submarine landslides have a characteristic interference pattern. 
This pattern could be then used to detect and characterize landslides. 

We start by a brief explanation about the interference pattern,  
then describe a simulation that reproduce the interference in a simplified case. 

\subsubsection{The Lloyd mirror effect}
When acoustic waves propagate in a bounded medium, the reflection on the boundary create interference. This effect is also called the Lloyd mirror effect \cite{jensen_computational_2011}. 
We consider a harmonic point source emitting rays in all directions, and neglect the reflection on the bottom.
Every point in the domain is connected to the source by two rays: the direct ray, and the ray reflected by the surface (see Figure  \ref{fig:LLoyd}). 
For points far away from the source, the pressure field is approximated by
\begin{equation} \label{eq:Simu3:Approx}
p(x,z,t) \sim - A \frac{2i}{x} \sin(k z_s \sin \theta) e^{ikx} e^{-i \omega t},
\quad \text{ for } x\gg z_s,
\end{equation}
where $A$  and $z_s$ are respectively the source magnitude and depth, $k$ is the wavenumber, $\theta$ is the declination angle, $\omega$ is the angular frequency.  
The pressure minima $|p_{\text{max}}| = 0$ are reached for 
\begin{equation}\label{eq:Simu3:InterMin}
\sin \theta = (m-1) \frac{\pi}{k z_s}, 
\quad \text{ with } m \in \mathbb N^*.
\end{equation}
Thanks to the relation \eqref{eq:Simu3:InterMin}, each point in space is associated to a ``frequency bandwidth", namely the distance between two frequencies corresponding to minimum pressure.
For a point with declination angle $\theta$, the associated bandwidth $\Delta f$ is given by
\begin{equation} \label{eq:Simu3:Bandwidth}
\Delta f = \frac{c}{2 z_s \sin \theta}.
\end{equation}
When recording the pressure at a fixed point and for a moving source (e.g. a landslide), the bandwidth $\Delta f$ should vary with time.
In \cite{caplan-auerbach_hydroacoustic_2014}, spectrograms computed from several fixed hydrophone show interference patterns. 
The measured bandwidths evolve with time, and are consistent with the theoretical value predicted by Equation \eqref{eq:Simu3:Bandwidth}.
These results suggest that the Lloyd mirror effect could be used to detect submarine landslides and to recover their velocity.

\begin{figure}
	\centering
    \includegraphics[width=0.6\textwidth]{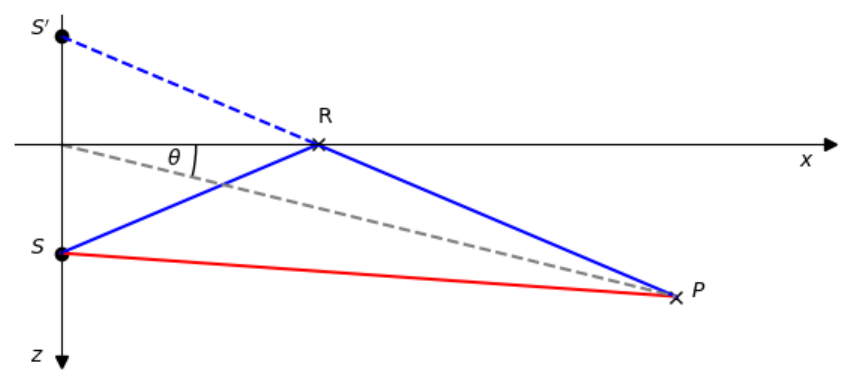}
	\caption{The direct path (in red) and the reflected path (in blue) connecting the source $S$ to a point $P$. The location of the reflection (point $R$) is deduced from the image source $S'$. The declination angle $\theta$ is also indicated.}
	\label{fig:LLoyd}
\end{figure}

\subsubsection{Simulation description and result}

\begin{figure}
	\centering
    \includegraphics[width=\textwidth]{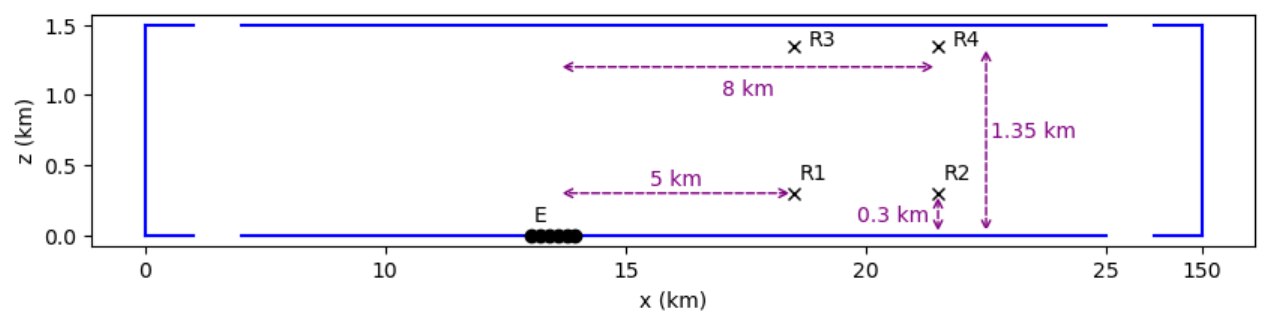}
	\caption{The domain for the third simulation: acoustic waves and tsunami generation by a submarine landslide. The source is in black and the four receivers are indicated with 'x'. } 
	\label{fig:Simu3:Domain}
\end{figure}
We reproduce here the interference pattern for a static emitter in the 2D case with a flat seabed. 
Since the source is static, the bandwidth should not change with time. 
To illustrate the dependency of the interference bandwidth with the emitter-receiver distance, we record the pressure at four different locations. 
The four receivers are denoted $R_i$ with $i \in \{1,2,3,4 \}$. 
The domain, with the emitter and receivers, is shown in Figure~\ref{fig:Simu3:Domain}.
The receivers locations and the corresponding theoretical bandwidth $\Delta f$, computed with Equation \eqref{eq:Simu3:Bandwidth}, are indicated in Table~\ref{table:Simu3:Receivers}.
The buoyancy frequency $N$ is constant in this scenario. 

It remains to define the source. 
Based on in-field data \cite{caplan-auerbach_hydroacoustic_2014}, the source is assumed to be a movement of the seabed over 600 m generating a continuous range of frequencies. 
In \cite{caplan-auerbach_hydroacoustic_2014} the frequencies range up to 400 Hz, but for computational reason, we restrain the frequencies to 20 Hz maximum. 
The Lloyd-mirror effect is still visible for those lower frequencies.
The function $u_b$ is given by $ u_b(x,t) = f(x) g(t) $,  where $f$ is a Gaussian function with magnitude $A$,
\[
f(x) = A \exp \left(-s_x^2 |x - x_0|^2 \right), 
\]
and $g(t)$ is a white noise with frequencies ranging from 0 to 20 Hz, see Figure~\ref{fig:Simu3:Emitter}.
For this simulation, the computational domain is long enough to avoid non-physical reflections. 
The simulation parameters are described in Table~\ref{tab:Simu3_Param}. \\

\begin{table}[]
    \centering
    \begin{tabular}{c | l || c | l}
    Parameter       & Value     & Parameter & Value    \\
    \hline
    Domain width    &  150 km       & $A$       & 1 ms$^{-1}$   \\
    Domain height   & 1.5 km        & $s_x$     & 0.07 m$^{-1}$ \\
    \verb|Px|       & 6             & $x_0$     & 75 km         \\
    \verb|Pz|       & 6             &           &               \\
    \verb|Nx|       & 600           &           &               \\
    \verb|Nz|       & 10            &           &                 
    \end{tabular}
    \caption{Parameter values for the third simulation: acoustic waves and tsunami generation by a submarine landslide.}
    \label{tab:Simu3_Param}
\end{table}

\begin{table}
\begin{center}
\begin{tabular}{  |c||c|c|c|}
\hline
 & x (km) & z (km) & Theoretical $\Delta f$ (Hz)
\\
\hline \hline 
$R_1$	& 5 & 0.3 & 2
\\
\hline
$R_2$& 8 & 0.3 & 3
\\
\hline
$R_3$& 5 & 1.35 & 16
\\
\hline
$R_4$& 8 & 1.35 & 26
\\
\hline
\end{tabular} 
\end{center}
\caption{Interference of acoustic waves generated by a submarine landlside: coordinates of the source center and the receivers, and theoretical $\Delta f$ obtained with Equation~\eqref{eq:Simu3:Bandwidth}.}
\label{table:Simu3:Receivers}
\end{table}

The spectrograms of the recorded pressures for each receiver are shown in Figure~\ref{fig:Simu3:Receivers}.
In the four spectrograms, we see that the frequency ranges from 0 to 20 Hz, which is consistent with the source spectrogram.
Constructive and destructive interferences are also clearly visible for each receiver.
For the receivers $R_1$ and $R_2$, both at $z=0.3$ km, the measured bandwidth is approximately 2 Hz.
The receivers $R_3$ and $R_4$ are both at $z=1.35$ km and their measured bandwidth is around 5 Hz. 
For the receivers near the surface ($R_3$ and $R_4$), the measured bandwidth do not correspond to the theoretical bandwidths presented in Table~\ref{table:Simu3:Receivers}. 
One reason for this difference could be that the theoretical value is computed with the assumption of negligible reflections on the seabed.
This assumption is probably not valid for our model, as the seabed is assumed rigid.
Even though the numerical values do not correspond, we note that the measured bandwidths share some qualitative properties with the theoretical values:
(1) the bandwidth is more sensitive to the depth of the hydro-acoustic sensor than to its distance to the landslide, and
(2) the bandwidth is larger for the receivers closer to the surface.

\begin{figure}
	\centering
    \includegraphics[width=\textwidth]{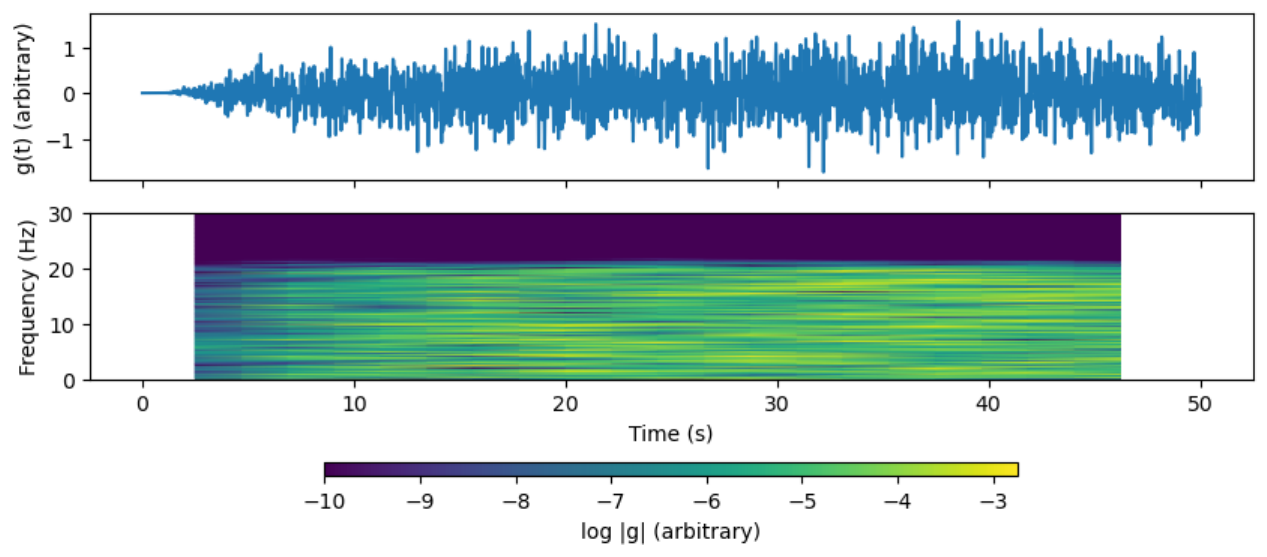}
	\caption{The function $g$ generating a uniform range of frequencies up to 20 Hz: time serie (top) and spectrogram of $g$ (bottom).
	}
	\label{fig:Simu3:Emitter}
\end{figure}
 
\begin{figure}
    \includegraphics[width=\textwidth]{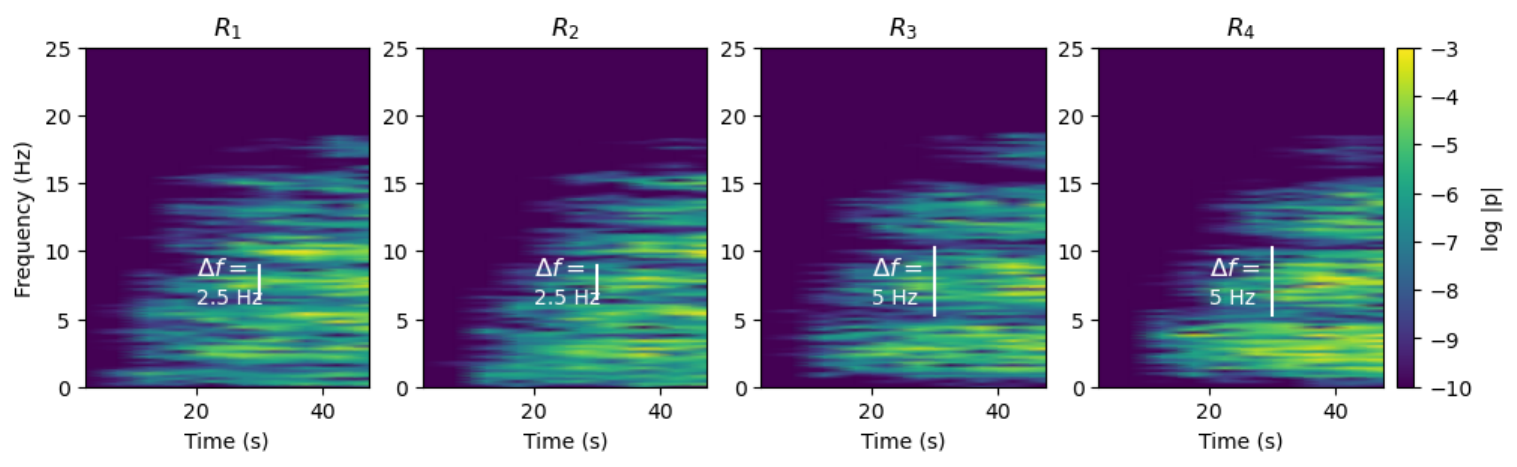}
	\caption{Interference of acoustic waves generated by an submarine landslide: each spectrogram corresponds to one receiver. The coordinates for each receiver are, in km: $R_1$=(5,0.3); $R_2$=(8,0.3); $R_3$=(5,1.35) and $R_4$=(8,1.35). The measured bandwidth is indicated in white. 
	} 
	\label{fig:Simu3:Receivers}
\end{figure}

In addition to the acoustic waves forming complex patterns, a tsunami is also generated. 
We show in Figure~\ref{fig:Simu3:Tsunami} the free-surface vertical displacement at a point 1.5 km away from the source center and the spectrogram of the signal. The tsunami is characterized on the spectrogram by the presence of a low frequency (below 5 Hz) component. We note that the low frequency is more prominent in the tsunami spectrogram than in the pressure spectrograms (see Figure \ref{fig:Simu3:Receivers}).
As in the previous tsunami simulation (see Figure \ref{fig:Comparison:Sammarco}), the curve shows a combination of high-frequency components starting at early times ( $t\sim 10$~s ) and a lower-frequency component starting at later time ($t\sim 40$~s). 
The high-frequency components correspond to the acoustic waves and the low-frequency one is the tsunami passing through the point.

\begin{figure}
	\centering
    \includegraphics[width=\textwidth]{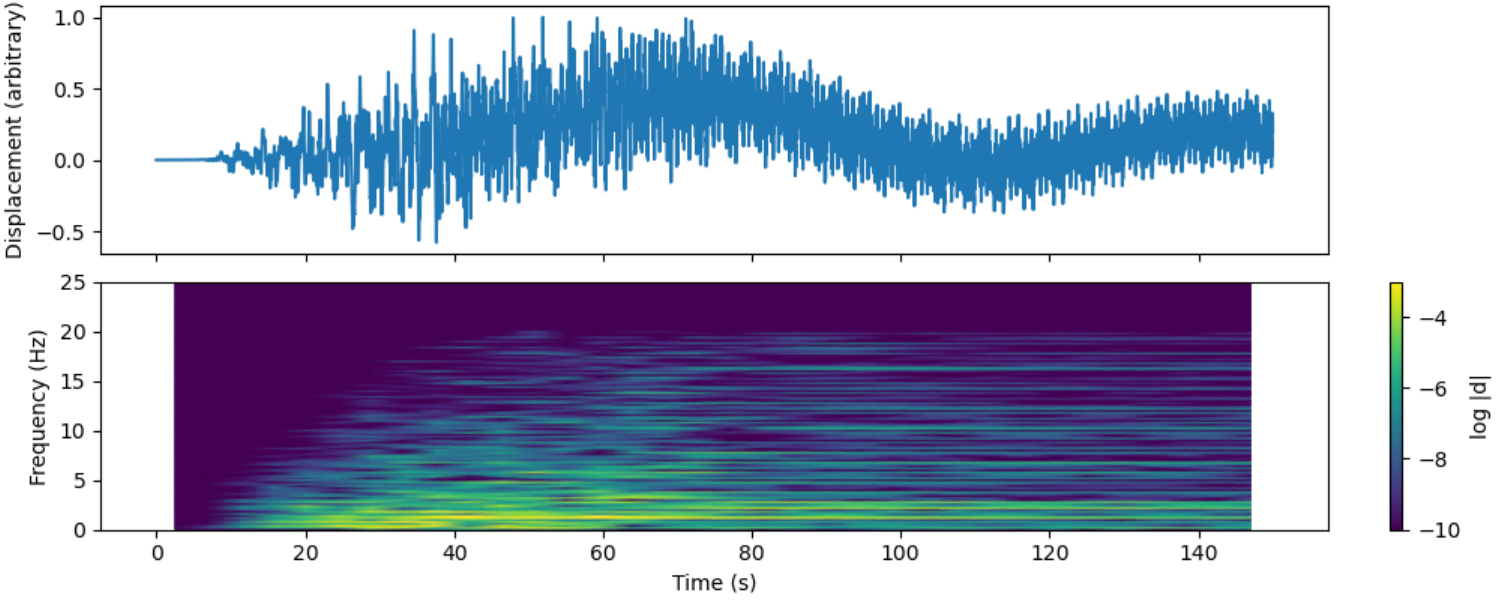}
	\caption{Tsunami generated by a submarine landslide: time series (top) and spectrogram (bottom) of the surface vertical displacement for a point $x=1.5$~km away from the source center.} 
	\label{fig:Simu3:Tsunami}
\end{figure}




\section{Conclusion and future work}
In this work, we have presented two different formulations for a model describing the propagation of acoustic-gravity waves in a stratified ocean.
The novel potential-based formulation is easier to handle from a mathematical point of view since it requires no essential boundary condition. 
Moreover, it offers several advantages for the numerical approximations: there are only 2 unknowns even for 3D problems, and the formulations includes Neumann-type boundary conditions, easier to implement than the boundary condition of the velocity-field formulation -- in particular when the seabed is not flat. 

We have first shown the well-posedness of both formulations. For the velocity-field formulation, the study was complicated by the presence of a non-homogeneous boundary condition of Dirichlet type. 
To show the existence of solution, an equivalent dissipative problem was introduced. 
The equivalence between the velocity-field and potential-based formulations was then proved by using energy estimates. 
A discretization of both formulations using the spectral element method was presented, and the schemes were validated on a two-dimensional numerical test case. 
The test case also illustrated the equivalence between both formulation. 
We then used the relation between the velocity and the generalized potential to study the error made by the hypothesis of an irrotational flow when the fluid is not barotropic.
Our numerical experiments suggest that the assumption is  well-justified, but that it is not valid uniformly in space and time.
Particular types of source can lead to a difference between the irrotational and the general model that increases in time on some very located portions of the domain.
Finally, we presented a simulation bringing to light the interference pattern caused by the surface reflection of hydro-acoustic waves created by submarine landslides. 
Generation of hydro-acoustic waves by landslides is a topic that has not been very much explored yet, even though the analysis of those hydro-acoustic waves may provide a unique tool to detect and characterize submarine landslides and the potential hazard related to generated tsunamis. 
Notably, numerical experiments are very scarce.

The potential-based formulation could offer a new approach for the study of the Galbrun equation. 
Indeed, as mentioned in the introduction, the velocity-field problem can be seen as the Galbrun equation with no mean flow. 
For the Galbrun equation with a mean flow, the choice of the correct functional framework for the analysis is still an open question. 
It could be interesting to investigate whether a potential-based formulation could be obtained for the Galbrun equation with a mean flow. 
Moreover, from a numerical point of view, the transient problem with a mean flow and the harmonic problem -- even without mean flow -- are known to present spurious modes when discretized in a simple way \cite{bonnet-ben_dhia_time-harmonic_2007, bonnet-ben_dhia_regularisation_2006}.
One could study the discretization of the potential-based formulation in the harmonic regime and check whether spurious modes are present.

Our simulations of a schematic landslide source and of the associated Lloyd mirror effect open new avenues to simulate in the same framework tsunamis and hydro-acoustic waves generated by landslides
in the linear approximation. 
Such simulations could ultimately provide insight into the most appropriate hydro-acoustic sensor configuration in the field, making it possible to detect these events.


\bibliographystyle{plain}
\bibliography{bib_analyse}


\section{Appendix}

\subsection{A realistic profile for the buoyancy frequency $N$} \label{sec:A:Buoyancy}
In \cite{dubois_acoustic_2023}, we have shown how profiles for $\rho(z), c(z)$ can be computed from a given temperature profile. 
For completeness we recall here the equations. 
The compressible Euler equations under the assumption of an initial state at equilibrium yield the following system of equations for the pressure $p_0$, temperature $T_0$ and density $\rho_0$, 
\begin{align}
\nabla p_0 = - \rho_0 g \, \e{z}, \quad 
\rho_0 = f_{\rho}(p_0, T_0), \quad
p_0(H) = p^a.
\label{eq:A:Equilibrium}
\end{align}
For the sound speed we use the expression $c_0=c(p_0,T_0)$ given in \cite{international_association_for_the_properties_of_water_and_steam_iapws_2009}. 
Hence, if $T_0(z)$ and an equation of state $f_\rho$ is given, the system
\begin{align}
\frac{d p_0}{d z} 
= -g f_{\rho}(p_0, T_0), 
& \quad  z \in (0, H),
\\
p_0 = p^a,  & \quad  z = H.
\label{eq:A:Hydrostat}
\end{align}
is solved to obtain $p_0$. Then $\rho_0$ and $c_0$ are computed from \eqref{eq:A:Equilibrium} and \cite{international_association_for_the_properties_of_water_and_steam_iapws_2009}. 
The equation of state $f_\rho$ is given in \cite{international_association_for_the_properties_of_water_and_steam_iapws_2009}. 
The differential equation for the pressure \eqref{eq:A:Hydrostat} is numerically solved for the temperature profile shown in Figure \ref{fig:Background_T}, and we obtain the profiles $\rho_0(z)$ and $c_0(z)$ shown in Figure \ref{fig:Background_rho} and \ref{fig:Background_C}. 
The buoyancy frequency $N$ is then computed from \eqref{def_of_N} and we obtain the profile shown in Figure \ref{fig:Background_N}.

\begin{figure}
\centering
\begin{subfigure}{0.255\textwidth}
	\includegraphics[width=\textwidth]{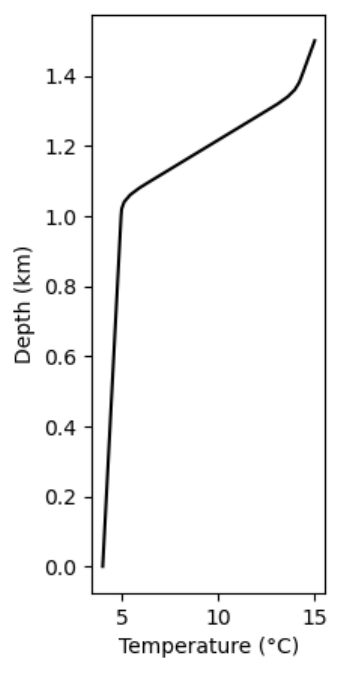}
	\caption{Temperature profile.}
	\label{fig:Background_T}
\end{subfigure}
\begin{subfigure}{0.20\textwidth}
	\includegraphics[width=\textwidth]{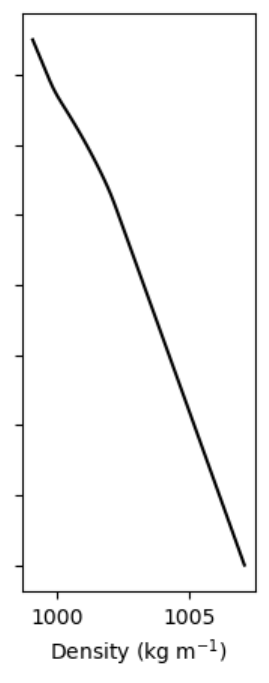}
	\caption{Density profile.}
	\label{fig:Background_rho}
\end{subfigure}
\begin{subfigure}{0.20\textwidth}
	\includegraphics[width=\textwidth]{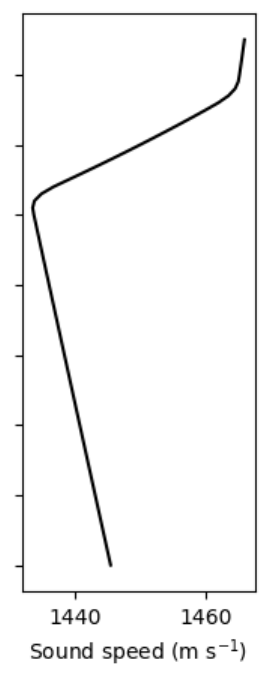}
	\caption{Sound speed profile.}
	\label{fig:Background_C}
\end{subfigure}
\begin{subfigure}{0.224\textwidth}
	\includegraphics[width=\textwidth]{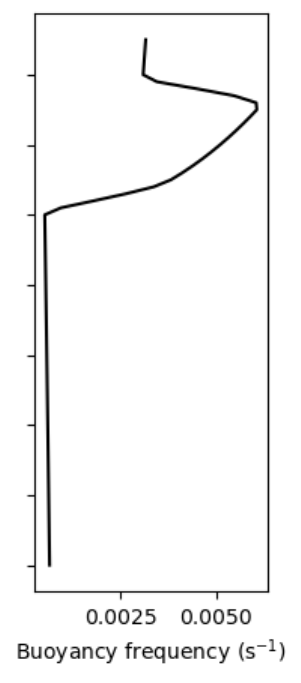}
	\caption{Buoyancy frequency.}
	\label{fig:Background_N}
\end{subfigure}
\caption{Profiles used in Section \ref{sec:Simu2}: (a)Temperature $T(z)$, (b) density $\rho_0(z)$ and (c) sound speed $c_0(z)$ used for the computation of (d) the buoyancy frequency $N(z)$.}
\label{fig:Background_profiles}
\end{figure}

\subsection{The irrotational flow assumption } \label{sec:A:Irrotational}
\begin{figure}
    \centering
    \includegraphics[width=\textwidth]{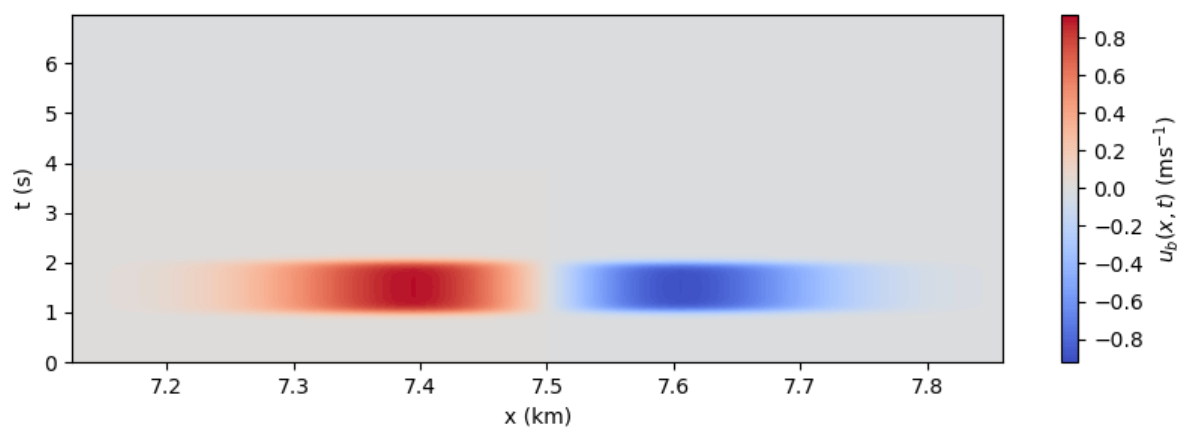}
    \caption{The source $u_b(x,t)$ over time and space.}
    \label{fig:A:Source}
\end{figure}

We study the same scenario as in Section \ref{sec:Simu2} but with a different function $g(t)$ representing the time dependency of the source $u_b$.
Here $g(t)$ is a smoothed rectangle,
\[
g(t) = \frac{1}{1+e^{- s_t (t-t_0)}}
- \frac{1}{1+e^{- s_t (t - t_0 - r_t)}},
\]
with $s_t = 20$~s$^{-1}$, $t_0=1$~s and $r_t=1$~s.
The difference with Section \ref{sec:Simu2} is that the time average of the source $u_b$ is not zero. 
The function $u_b(x,t)$ is shown in Figure \ref{fig:A:Source}.
The simulation uses the topography described by Equation \eqref{eq:Simu2:Topo}  and the profile $N(z)$ described in Section \ref{sec:A:Buoyancy}.
Snapshots of the remainder $|U_r|$ (see Figure \ref{fig:A:Ratio}) show that even though being small, the values near the source increase with time.  
\begin{figure}
    \centering
    \begin{subfigure}{\textwidth}
	\includegraphics[width=\textwidth]{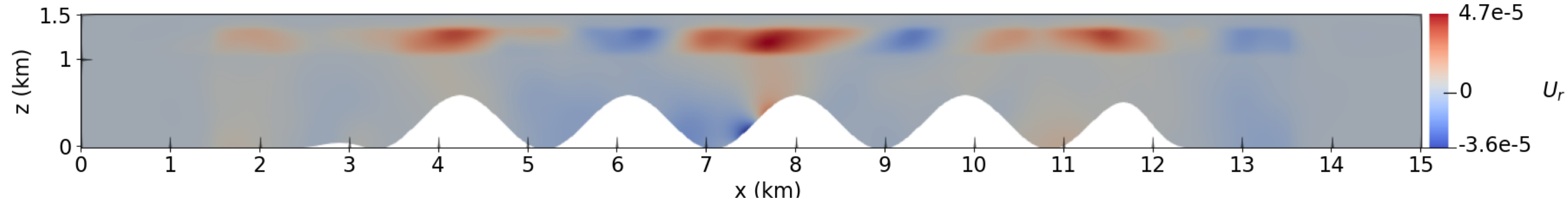}
	\caption{$t=8$~s}
    \end{subfigure}
    \begin{subfigure}{\textwidth}
	\includegraphics[width=\textwidth]{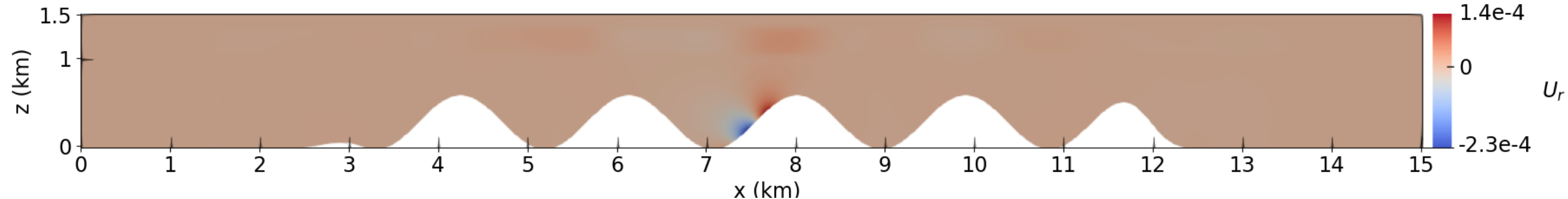}
	\caption{$t=60$~s}
    \end{subfigure}
    \begin{subfigure}{\textwidth}
	\includegraphics[width=\textwidth]{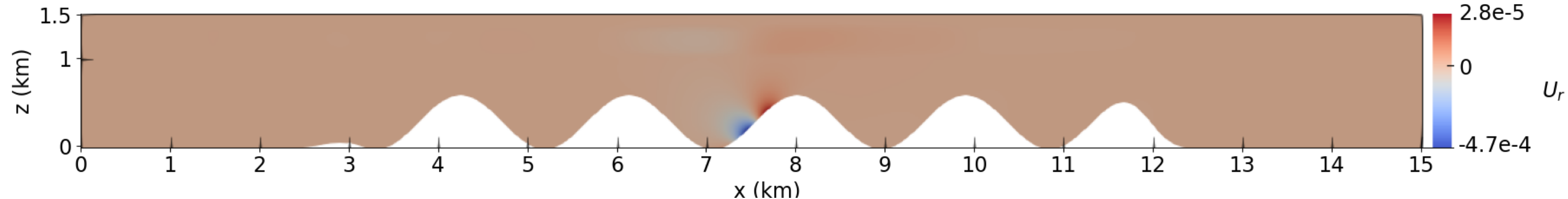}
	\caption{$t=120$~s}
    \end{subfigure}
    \caption{Snapshots of $U_r$ at various times: (a) $t=8$~s , (b) $t=60$~s and (c) $t=120$~s. Note how the color scale changes at each snapshot. }
    \label{fig:A:Ratio}
\end{figure}

\end{document}